\documentclass{mjm2}

\title{On an Index Theorem of Chang, Weinberger, and Yu}
\headtitle{On an Index Theorem of Chang, Weinberger, and Yu}
\author{Thomas
  Schick}
\headauthor{Thomas Schick and Mehran Seyedhosseini}
\address{Thomas Schick\\  Mathematisches Institut\\
  Universit\"at G\"ottingen\\
  Germany}
\email{thomas.schick@math.uni-goettingen.de}
\urladdr{http:www.uni-math.gwdg.de/schick}

\author{Mehran Seyedhosseini}
\address{Mehran Seyedhosseini\\
  Institut für Mathematik\\
  Universität Potsdam\\
  Germany}
\email{mehran.seyed.hosseini@uni-potsdam.de}
\urladdr{https://www.math.uni-potsdam.de/professuren/geometrie/personen/mehran-seyedhosseini/}

\usepackage{mjm2}
\usepackage{hyperref}

\usepackage{color}
\usepackage[utf8]{inputenc}
\usepackage{amsmath}
\usepackage{amsthm}
\usepackage{amsrefs}
\usepackage{amssymb}
\usepackage{mathrsfs}
\usepackage{amscd}
\usepackage{mathtools}
\usepackage{tikz-cd}
\usepackage{verbatim}
\usepackage{slashed}

\newcounter{commentcounter}

\usepackage{color}
\usepackage{ifthen,srcltx}
\newcommand{\showcomments}{yes}

\newsavebox{\commentbox}
\newenvironment{com}%
{\ifthenelse{\equal{\showcomments}{yes}}%
{\footnotemark
        \begin{lrbox}{\commentbox}
        \begin{minipage}[t]{1.25in}\raggedright\sffamily\tiny
        \footnotemark[\arabic{footnote}]}
{\begin{lrbox}{\commentbox}}}
{\ifthenelse{\equal{\showcomments}{yes}}
{\end{minipage}\end{lrbox}\marginpar{\usebox{\commentbox}}}
{\end{lrbox}}}

\newtheorem{theorem}[subsection]{Theorem}
\newtheorem{proposition}[subsection]{Proposition}
\newtheorem{lemma}[subsection]{Lemma}
\theoremstyle{definition}
\newtheorem{definition}[subsection]{Definition}
\theoremstyle{definition}
\newtheorem{remark}[subsection]{Remark}
\newtheorem{corollary}[subsection]{Corollary}
\DeclareMathOperator{\supp}{supp}

\DeclareMathOperator{\prop}{prop}
\DeclareMathOperator{\kernel}{ker}
\DeclareMathOperator{\ind}{ind}
\DeclareMathOperator{\red}{red}
\DeclareMathOperator{\maxi}{max}
\DeclareMathOperator{\Max}{m}
\DeclareMathOperator{\ad}{Ad}
\DeclareMathOperator{\ev}{ev}
\DeclareMathOperator{\Ind}{Ind}
\newcommand{\reals}{\mathbb{R}}
\newcommand{\integers}{\mathbb{Z}}
\newcommand{\norm}[1]{||#1||}
\newcommand{\normalsubgroup}{\lhd}
\newcommand{\iso}{\cong}

\begin{document}

\begin{abstract}
In this paper we prove a strengthening of a theorem of
Chang, Weinberger and Yu on obstructions to the existence of positive scalar
curvature metrics on compact 
manifolds with boundary. They construct a relative index for the Dirac operator, 
which lives in a relative $K$-theory group, measuring the difference between the
fundamental group of the boundary and of the full manifold.

Whenever the
  Riemannian metric has product structure and
  positive scalar curvature near the boundary, one can define an absolute index of the
  Dirac operator taking value in the K-theory of the $C^*$-algebra of
  fundamental group of the full manifold. This index depends on the metric
  near the boundary. We prove that (a slight variation of) the relative index of Chang,
  Weinberger and Yu is the image of this absolute index under the canonical
  map of K-theory groups.

  This has the immediate corollary that positive
  scalar curvature on the whole manifold implies vanishing of the relative
  index,  giving a conceptual and direct proof of the vanishing
  theorem of Chang, Weinberger, and Yu (rather: a slight variation). To take the fundamental groups of the
  manifold and its boundary into account requires working with maximal
$C^*$-completions of the involved $*$-algebras. A significant part of this
paper is devoted to foundational results regarding these completions. On the
other hand, we introduce and propose a more conceptual and more geometric
completion, which still has all the required functoriality.
\end{abstract}

\maketitle

\section{Introduction}
In \cite{CWY} Chang, Weinberger and Yu define a relative index of the Dirac
operator on a compact spin manifold $M$ with boundary $N$ as an element of
$K_*(C^*(\pi_1(M),\pi_1(N)))$, where this relative K-theory group measures the
difference between the two fundamental groups. The main geometric theorem of
\cite{CWY} then says that the
existence of a positive scalar curvature metric on $M$ which is collared at
the boundary implies the vanishing of this index. The argument for this
vanishing theorem is rather complicated and indeed contains a gap. We address
this gap in this paper. After the first version of the present article was
made public, \cite{GuoXieYu} was posted, which also attempts to fix this
gap.

More explicitly, the $K$-theory groups of the absolute and relative group $C^*$-algebras of the manifold and its boundary fit in a long exact sequence
\begin{equation}
 \rightarrow K_*(C^*(\pi_1(N))) \rightarrow K_*(C^*(\pi_1(M)))
\xrightarrow{j} K_*(C^*(\pi_1(M),\pi_1(N))) \rightarrow 
\label{eq:relseq}
\end{equation}
 The relative index $\mu([M,N])$ is defined as the image of a relative fundamental class
 $[M,N]\in K_{\dim M}(M,N)$ under a relative index map $\mu\colon K_*(M,N)\to
 K_*(C^*(\pi_1(M),\pi_1(N)))$. Here, $K_*(M,N)$ is the relative K-homology and
 $[M,N]$ is constructed with the help of the Dirac operator on $M$. Indeed, in
 this paper we mainly deal with a small variant of the construction of
 \cite{CWY} by choosing a slightly different $C^*$-completion. We discuss
 this in more detail below, throughout the introduction, we work with this
 modification. 

Our main goal is to better understand the vanishing theorem of Chang,
Weinberger and Yu, and to prove a
strengthening of it, at the same time giving a new and more conceptual
proof.

For our approach, recall that one has a perfectly well defined  K-theoretic
index of the
Dirac operator on a Riemannian manifold with boundary provided the boundary
operator is
invertible, for example if the metric is collared and of positive scalar
curvature near the boundary (see e.g.~\cite{PSc}). This index takes values in
$K_*(C^*(\pi_1(M)))$
and explicitly depends on the boundary operator (i.e.~on the positive scalar
curvature metric $g$ of the boundary). In the latter case we denote it by
$\Ind^{\pi_1(M)}(g)\in K_*(C^*(\pi_1(M)))$. Our main result states
that a slight variant of the relative index of 
Chang-Weinberger-Yu is the image of the absolute index defined with invertible
boundary operator under the natural homomorphism $j$ of~\eqref{eq:relseq}
(whenever this absolute index is defined):
\begin{theorem}\label{theo:main_intro}
  \begin{equation*}
  j(\Ind^{\pi_1(M)}(g)) = \mu([M,N]).
\end{equation*}
\end{theorem}

The absolute index $\Ind^{\pi_1(M)}( g)$ vanishes
whenever we have positive scalar curvature on all of $M$, implying immediately
the corresponding vanishing result for the relative index of Chang,
Weinberger, and Yu.



Relative index theory has recently been the subject of considerable
activity. In \cite{DG}, Deeley and Goffeng define a relative index map using geometric 
K-homology instead of
coarse geometry and prove index and  vanishing results similar to the main result of
our paper. However, this relies and uses the full package of higher
Atiyah-Patodi-Singer index theory (like \cite{LeichtnamPiazza}), which we
consider technically very
demanding and somewhat alien to the spirit of large scale index
theory. Indeed, in \cite{DG} it is not even proved in general that the
constructions coincide with the ones of \cite{CWY}. 
Yet another approach to relative index theory and the results of \cite{CWY} is
given by Kubota in \cite{Kubota}. There, the new concepts of relative Mishchenko bundles and
Mishchenko-Fomenko index theory are introduced, and heavy use is made of the
machinery of KK-theory. In \cite{Kubota}, a careful identification of the
different approaches is carried out.

The main point of our paper is its very direct and rather
easy approach to the index theorems as described above. We work entirely in the
realm of large scale index theory, and just rely on the basic properties of
the Dirac operator (locality, finite propagation of the wave operator,
ellipticity). We avoid APS boundary conditions and we avoid deep
KK-techniques. Such a direct approach is relevant also 
because it is more likely to allow for the construction of secondary
invariants, to be used for classification rather than obstruction purposes.

In \cite{CWY}, fundamental use is made of the
\emph{maximal} Roe and localisation algebras to obtain the required
functoriality needed e.g.~in the sequence \eqref{eq:relseq}. The
identification of its K-theory with K-homology of the space is needed for the
\emph{maximal} localisation algebra and reference is given to \cite{QR} for the
proof. However, that reference only deals with the \emph{reduced}
setting. Working out the details to extend the known
results to the maximal setting turned out to be rather non-trivial. The first
part of the present paper is devoted to the careful development of
foundational issues of maximal Roe and localisation algebras.
For us, this complete and careful
discussion of the properties of maximal completions in the context of coarse
index theory  is the second main contribution of this paper. Our results on
this are used e.g.~in \cite{DG}.

The maximal Roe algebra is defined in a rather ad hoc and ungeometric way: one
comes up with the (somewhat arbitrary) algebraic Roe algebra, a $*$-subalgebra
of bounded operators on a Hilbert space which is \emph{not} closed, and then
passes to the maximal $C^*$-closure. This is hard to control and to compute
(there are very few cases of actual computation), and geometric arguments are
very delicate. It required the whole additional unpublished preprint
\cite{GuoXieYu}, which appeared after the first version of this paper was posted, to prove the claim of \cite{CWY} that the
Schr\"odinger-Lichnerowicz vanishing theorem applies also to in the maximal
Roe algebra. This claim was unjustified in \cite{CWY}, as the authors of
\cite{GuoXieYu} also observe.

Our approach is going in a different direction. We propose to use instead of
the ad hoc maximal completion a much more geometric completion $C^*_q$, which we
introduce in Section \ref{sec:quotient_completion}. Problems with the
standard (reduced) Roe algebra arise in the equivariant setting of the group
$\Gamma$ acting on the space $X$ due to lack of functoriality. Our completion
takes all normal quotients $\Gamma/N$ 
acting on $X/N$ into account. This restores full functoriality, but is
completely geometric. The Schr\"odinger-Lichnerowicz formula and other
geometric arguments apply effortless.

The precise formulation of Theorem \ref{theo:main_intro} and of
\eqref{eq:relseq} requires to specify which completion is used. In our
approach, this becomes
$C^*_q(\pi_1(M),\pi_1(N))$, involving the completions of the group
algebras in the direct sum of the regular representation of all its 
quotients. Formally, the relative index in this K-theory group is weaker than the relative index obtained by using the maximal completion. However, not a single case is known where
extra information on obstructions and classification has been obtained from
the difference of the K-theory of the maximal and the reduced group
$C^*$-algebras, and the Novikov conjecture suggests that this should not be
possible. In any event, it seems extremely hard to exploit such a difference
for geometric means. So we believe that our approach and our completion is
a very good choice: full functoriality, no extra effort for geometric
arguments, in practice no loss of information.

\begin{remark}
  Our approach works for arbitrary, also non-cocompact situations. In the
  cocompact case, there is another way for geometric constructions: one works
  with the compact space, and with the infinite dimensional Mishchenko
  bundle. Here, one has the choice to use arbitrary group algebra completions,
  including the maximal one, which is used in \cite{DG} and \cite{Kubota}.
\end{remark}

\begin{remark}
  We present details of the construction and manipulation of the relative
  index and the vanishing theorem only in the case that the dimension of the
  manifold is even. We chose to do this because this is the most classical
  setup, and the constructions are particularly explicit and direct. This also
  means that  we remain close to the original treatment of \cite{CWY}.

  We discuss in Remark \ref{rem:odd_dimensions} how one can reduce the general
  case 
  to the even dimensional situation. We also discuss there how one could use
  the techniques of Zeidler \cite{RZA} combined with our setup to uniformly
  treat all dimensions and even the case of real $C^*$-algebras.
\end{remark}

In parts of the present paper we give missing arguments for some of the
results of the master thesis of 
Seyedhosseini \cite{Seyedhosseini}.

\subsection{Structure of the paper}

In Section \ref{sec:maxroe} we present our foundational results on \emph{maximal} Roe
algebras. In Section \ref{sec:quotient_completion}, we introduce our geometric 
functorial completed Roe algebra and establish its main properties. Section \ref{relindd} recalls the construction
of the relative index, following \cite{CWY}. We try to motivate the
construction, give additional details and fix small glitches in \cite{CWY}. Section
\ref{sec:proof} gives the proof of Theorem \ref{theo:main_intro}.

\medskip

\noindent\textbf{Acknowledgement}.
We would like to thank Ralf Meyer for useful discussions and the referee for
comments helping to improve the presentation.

\section{The Maximal Roe Algebra}
\label{sec:maxroe}
In the following, we will only consider separable and proper metric spaces with bounded geometry. We recall that a locally compact metric
$X$ space has bounded geometry if one can find a discrete subset $Y$ of $X$ such that:
\begin{itemize}
 \item There exists $c > 0$ such that every $x \in X$ has distance less than $c$ to some $y \in Y$.
 \item For all $r > 0$ there is $N_r$ such that  $\forall x \in X$ we have
   $|Y \cap B_r(x)|\le N_r$.
 \end{itemize}

A covering of a compact Riemannian manifold with the lifted metric obviously
has bounded geometry.
 
\subsection{Roe Algebras}\label{sec:max_fundationas}
Let $X$ be a separable and proper mertric space endowed with a free and proper action of a discrete group $\Gamma$ by isometries. In this section, we will recall the
definition of the Roe algebra associated to $X$. Let $\rho \colon C_0(X)
\rightarrow L(H)$ be an ample, nondegenerate representation of $C_0(X)$ on some separable
Hilbert space $H$. A representation of $C_0(X)$ is called ample if no non-zero element of $C_0(X)$ acts as a
compact operator on $H$. The representation $\rho$ is called
covariant for a unitary representation $\pi \colon \Gamma \rightarrow U(H)$ of
$\Gamma$ if $\rho(f_\gamma) = \ad_{\pi(\gamma)}\rho(f)\; \forall\gamma
\in \Gamma$. 
Here $f_\gamma$ denotes the function $x \mapsto f(\gamma^{-1}x)$.

From now on we will assume that $\rho$ is an ample and covariant representation of $C_0(X)$ as above. By an abuse of notation we will denote $\rho(f)$ simply
by $f$. We will later use representations of $C_0(X)$ which are an infinite direct sum of copies of an ample representation. Such representations are called very ample.  

\begin{definition}
 An operator $T \in L(H)$ is called a finite propagation operator if there
 exists an $r > 0$ such that $fTg = 0$ for all those $f,g \in C_0(X)$ with the
 property
 $d(\supp(f),\supp(g)) \ge r$. The smallest such $r$ is called the propagation of $T$ and is denoted by $\prop T$. An operator $T \in L(H)$ is called locally
 compact if $Tf$ and $fT$ are compact for all $f \in C_0(X)$.
\end{definition}

\begin{definition}
 Denote by $\mathbb{R}_\rho(X)^{\Gamma}$ the $*$-algebra of finite propagation, locally compact operators in $L(H)$ which are furthermore invariant under the
 action of the group $\Gamma$. We will call $\mathbb{R}_\rho(X)^{\Gamma}$ the algebraic Roe algebra of $X$. The maximal Roe algebra associated to the space
 $X$ is the maximal $C^*$-completion of $\mathbb{R}_\rho(X)^{\Gamma}$, i.e.\
 the completion of $\mathbb{R}_\rho(X)^{\Gamma}$ with respect to the supremum
 of all $C^*$-norms. This supremum is finite for spaces of bounded geometry by
 Proposition \ref{prop:max_exists}. It will be denoted by $C^*_{\rho,\maxi}(X)^\Gamma$. The reduced
 Roe algebra is the completion of the latter $*$-algebra using the norm in $L(H)$. We denote this algebra by $C^*_{\rho,\red}(X)^\Gamma$.
\end{definition}

\begin{proposition}\label{prop:max_exists}
\label{max_well_def}
  Suppose $X$ has bounded geometry. For every $R>0$ there is a constant $C_R$ such that for every $T\in
  \mathbb{R}_{\rho}(X)^\Gamma$ with propagation less than $R$ and every $*$-representation
  $\pi\colon \mathbb{R}(X)^\Gamma\to L(H^\prime)$ we have
  \begin{equation*}
    || \pi(T)||_{L(H^\prime)} \le C_R ||T||_{C^*_{\rho,\red}(X)^\Gamma}.
  \end{equation*}
  In particular, $||T||_{C^*_{\rho,\maxi}(X)^\Gamma}\le C_R||T||_{C^*_{\rho,\red}(X)^{\Gamma}}$
  and the bounded geometry assumption on $X$ implies that the maximal Roe
  algebra is well-defined.
\end{proposition}
\begin{proof}
 This follows from \cite{GongWangYu}*{Lemma 3.4} and \cite{EM}*{Theorem 2.7}.
\end{proof}

  Note that Proposition \ref{max_well_def} implies that restricted to the subset of operators of propagation bounded by $R$, the reduced and the maximal norms are equivalent.

\begin{proposition}\label{prop:Roe_alg_well_def}
 The $K$-theory groups of the reduced and maximal Roe algebra are independent of the chosen ample and covariant
 representation up to a canonical isomorphism.
\label{kind}
 \end{proposition}
\begin{proof}
 In the reduced case, this is the content of \cite{HR}*{Corollary 6.3.13}. For the maximal case we just note that conjugation by the
 isometries of the kind handled in
 \cite{HR}*{Section 6.3} gives rise to $*$-homomorphisms of the
 \textit{algebraic} Roe algebra and thus extend to morphisms of the maximal
 Roe algebras. Up to stabilisation, any two such morphisms can be obtained from each other by conjugation by a unitary
 making the induced map in K-theory canonical.
\end{proof}

\begin{remark}
 As a consequence of Proposition~\ref{kind} we will drop $\rho$ in our notation for the Roe algebras. Later we will introduce a new completion of
 $\mathbb{R}(X)^\Gamma$, which sits between the reduced and maximal completions and denote it by $C^*_{q}(X)^\Gamma$. Moreover, if $\Gamma$ is the trivial group, we will denote the Roe algebra by $C^*_{d}(X)$, where $d$ stands for the chosen completion.
\end{remark}

\begin{proposition}
 The $K$-theory of the maximal Roe algebra is functorial for coarse maps
 between locally compact metric spaces.
\end{proposition}
\begin{proof}
 The proof is similar to that of Proposition~\ref{kind} and makes use of
 it. In the reduced case, this is proved by constructing an appropriate
 isometry between the  representation spaces. Conjugation with the latter
 isometry gives rise to a $*$-homomorphisms of the
 \textit{algebraic} Roe algebra and thus extends to a morphism of the reduced and maximal Roe algebra. The latter then gives rise to
 homomorphisms of the $K$-theory groups of the Roe algebra. As in the proof of Proposition
 \ref{prop:Roe_alg_well_def}, the induced map in K-theory is canonical which
 also implies functoriality. See \cite{HR}*{Section 6.3} for a more detailed discussion. 
\end{proof}

In the case where $\Gamma$ acts cocompactly on $X$, we have the following
theorem. 
\begin{theorem} \label{eg}
 Suppose that $\Gamma$ acts cocompactly on $X$. Then $K_*(C^*_{\maxi}(X)^\Gamma) \cong K_*(C^*_{\maxi}(\Gamma))$.
\end{theorem}
\begin{proof}
 See \cite{GongWangYu}*{Section 3.12 \& 3.14} for the isomorphism $C^*_{\max}(|\Gamma|)^{\Gamma} \cong C^*_{\max}(\Gamma) \otimes K(H)$, where $C^*_{\max}(|\Gamma|)^{\Gamma}$ is the equivariant Roe algebra of $\Gamma$ seen as a metric space using some word metric. The action of $\Gamma$ on itself is given by left multiplication. Since the action of $\Gamma$ on $X$ is cocompact, the $\Gamma$-space $X$ is coarsely equivalent to $\Gamma$. This implies that $K_*(C^*_{\max}(X)^\Gamma) \cong K_*(C^*_{\max}(\Gamma)^\Gamma)$. The claim then follows from the stability of $K$-theory. 
\end{proof}

For a $\Gamma$-invariant closed subset $Y$ of $X$, we would like to define its
Roe algebra relative to $X$ as a closure of a space of operators in
$C^*_{\maxi}(X)^\Gamma$, which are suitably supported near $Y$. The next two definitions make this precise.
\begin{definition}
 For an operator $T \in L(H)$ we define the support $\supp T$ of $T$ as the complement of the union of all open sets $U_1\times U_2 \subset X \times X$ with the
 property that $fTg = 0$ for all $f$ and $g$ with $\supp f \subset U_1$ and $\supp g \subset U_2$. $T$ is said to
 be \emph{supported near $Y \subset X$} if there exists $r > 0$ such that $\supp T \subset B_r(Y) \times B_r(Y)$. Here and afterwards $B_r(Y)$ denotes the open
 $r$-neighbourhood of $Y$.
 \label{supp}
\end{definition}
\begin{definition}
 For a $\Gamma$-invariant closed subset $Y$ of $X$ as above, denote by $\mathbb{R}(Y\subset X)^\Gamma$ the $*$-algebra of operators in $\mathbb{R}(X)^\Gamma$
 which are supported near $Y$. The relative Roe algebra of $Y$ in $X$ is defined as the closure of $\mathbb{R}(Y\subset X)^\Gamma$
 in $C^*_{\maxi}(X)^\Gamma$ and is an ideal inside the latter $C^*$-algebra. It is denoted by
 $C^*_{\maxi}(Y \subset X)^\Gamma$.
 \label{relroe}
\end{definition}
Since $Y$ is a locally compact metric space with an action of $\Gamma$, it has its own (absolute) equivariant Roe algebra
$C^*_{\maxi}(Y)^\Gamma$. Theorem~\ref{theo:abs_vers_rel} identifies the
$K$-theory of the relative and absolute equivariant Roe
algebras in the case, where the action of $\Gamma$ on the subset is cocompact. However, for its proof we need further
conditions on the group action. 
\begin{definition}
 Let $\Gamma$ act freely and properly by isometries on $X$. $\Gamma$ is said
 to \emph{act conveniently} if there exists a fundamental domain $F$ for the action of $\Gamma$ satisfying:
 \begin{itemize}
 \item For each $R>0$, there exist $\gamma_1,\hdots \gamma_{N_R} \in \Gamma$ such that $B_R(F) \subset \bigcup_{i=1}^{N_R} \gamma_i\cdot F$
  \item For each $\gamma \in \Gamma$ and $R > 0$ there exists $S(R,\gamma) >
    0$ such that $\gamma^{-1}B_R(x) \cap F \subset B_{S(R,\gamma)}(x)$ fof all
    $x\in F$.
 \end{itemize}

\end{definition}

\begin{theorem}\label{theo:abs_vers_rel}
 Let $Y$ and $X$ be as above and suppose that $\Gamma$ acts conveniently on
 $X$ and cocompactly on $Y$. The inclusion
 $Y \rightarrow X$ induces an isomorphism
 $K_*(C^*_{\maxi}(Y)^\Gamma) \cong K_*(C^*_{\maxi}(Y \subset X)^\Gamma)$.
 \label{relroei}
\end{theorem}

\begin{remark}
  A representation $\rho\colon C_0(X) \rightarrow L(H_X)$ gives rise to a
  spectral measure which can be used to extend $\rho$ to the $C^*$-algebra
  $B_\infty(X)$ of bounded Borel functions on $X$ (see \cite{MUR}*{Theorem
    2.5.5}). Given $Y \subset X$, we get a representation
  $C_0(Y) \rightarrow L(\chi_YH_X)$. This is what is meant in the following
  Lemma~\ref{relext} by ``compressing the representation space of $C_0(X)$ in
  order to obtain a representation of $C_0(Y)$''. Given $Y$ as above we can choose $\rho$ such that it and its compression to $Y$ are both ample; for example, by choosing the ample representation of $X$ to be given by multiplication of functions with square summable sequences on some countable dense subset of $X$ whose intersection with $Y$ is a dense subset of $Y$. We will need
  Lemma~\ref{relext} for the proof of Theorem \ref{theo:abs_vers_rel}. Indeed,
  the novel difficulty in Theorem \ref{theo:abs_vers_rel} is to relate
  the $*$-representations used in the definition of $C^*_{\maxi}(Y)^\Gamma$
  with the 
 $*$-representations used to define $C^*_{\maxi}(X)^\Gamma$ ---of which
 $C^*_{\maxi}(Y\subset X)^\Gamma$ by definition is an ideal. Note that, at the
 moment, we only manage to do this if $Y$ is cocompact and the $\Gamma$-action is convenient.
 It is an interesting challenge to generalise Theorem \ref{theo:abs_vers_rel} to arbitrary pairs
 $(X,Y)$ and arbitrary free and proper actions. 
 \label{rem:difficmaxred}
\end{remark}
In the following lemma we choose an ample representation of $X$ which can be compressed to an ample representation of $Z$. As pointed out in the previous remark, this can always be done.
\begin{lemma}
 Let $\Gamma$ act conveniently on $X$ and $Z \subset X$ be $\Gamma$-invariant and suppose that the action of $\Gamma$ on $Z$ is cocompact.
 Construct $\mathbb{R}(Z)^{\Gamma}$ by compressing the
 representation space of $C_0(X)$, so that $\mathbb{R}(Z)^{\Gamma}$ is naturally a $*$-subalgebra of $\mathbb{R}(X)^{\Gamma}$. Then an arbitrary
 non-degenerate $*$-representation of $\mathbb{R}(Z)^{\Gamma}$ on a Hilbert space can be extended to a non-degenerate $*$-representation of
 $\mathbb{R}(X)^{\Gamma}$. In particular, the inclusion $\reals(Z)^\Gamma \rightarrow \reals(X)^\Gamma$ extends to an injection $C^*_{\max}(Z)^\Gamma \rightarrow C^*_{\max}(X)^\Gamma$. 
 \label{relext}
\end{lemma}
\begin{proof}
 Choose an ample representation $\rho\colon C_0(X) \rightarrow L(H_X)$ as above. By compressing the Hilbert space
 $H_X$ and restricting the representation, we obtain an ample
 representation of $C_0(Z)$, i.e.\
 $\rho|_{C_0(Z)}\colon C_0(Z) \rightarrow L(H_Z)$, where $H_Z$
 denotes the space $\chi_ZH_X$. Choose $D_Z \subset D_X$ fundamental domains
 of $Z$ and $X$ for the action of $\Gamma$. Similarly to the proof of
 \cite{HR}*{Lemma 12.5.3} one has
 $\mathbb{R}(Z)^\Gamma \cong \mathbb{C}[\Gamma] \odot K(\tilde{H}_Z)$, where $\tilde{H}_Z = \chi_{D_Z}H_Z$. The latter isomorphism is obtained using the
 isomorphisms $H_Z \cong \bigoplus_{\gamma \in \Gamma} \tilde{H}_Z \cong l^2(\Gamma) \otimes \tilde{H}_Z$. Denote by $\tilde{H}_X$ the Hilbert space
 $\chi_{D_X}H_X$. The isomorphism constructed in the proof can be extended to an injective map
 $\mathbb{C}[\Gamma] \odot L(\tilde{H}_X) \rightarrow L(H_X)$. The convenience of the action implies that its image contains the algebra
 $\mathbb{F}(X)^{\Gamma}$ of finite propagation $\Gamma$-invariant operators on $X$. This injection makes the diagram
 $$\begin{tikzcd}
    \mathbb{C}[\Gamma] \odot K(\tilde{H}_Z) \arrow{r}{\cong} \arrow[hook]{d} & \mathbb{R}(Z)^\Gamma \arrow[hook]{d}\\
    \mathbb{C}[\Gamma] \odot L(\tilde{H}_X) \arrow{r} & L(H_X)
 \end{tikzcd}$$
 commutative.
 We show that an arbitrary non-degenerate $*$-representation of $\mathbb{C}[\Gamma] \odot K(\tilde{H}_Z)$ on a Hilbert space $H_0$ can be extended to a
 non-degenerate $*$-representation of $\mathbb{C}[\Gamma] \odot L(\tilde{H}_X)$. This implies the
 lemma since $\mathbb{R}(X)^\Gamma \subset \mathbb{F}(X)^\Gamma$.
 Suppose that $\pi\colon \mathbb{C}[\Gamma] \odot K(\tilde{H}_Z) \rightarrow L(H_0)$ is a non-degenerate $*$-representation of $\mathbb{C}[\Gamma] \odot K(\tilde{H}_Z)$ on a
 Hilbert space $H_0$. The representation $\pi$ extends to a representation of $C^*_{\maxi}(\Gamma) \otimes K(\tilde{H}_Z)$ which we denote by $\pi$. Note that
 since the $C^*$-algebra of compact operators is nuclear, the $C^*$-algebra tensor product above is unique. $C^*_{\maxi}(\Gamma) \otimes K(\tilde{H}_Z)$ is a
 $C^*$-subalgebra of $C^*_{\maxi}(\Gamma) \otimes K(\tilde{H}_X)$ and $\pi$ can thus be extended to a non-degenerate representation of
 $C^*_{\maxi}(\Gamma) \otimes K(\tilde{H}_X)$ on a possibly bigger Hilbert space $H$, which we denote by $\tilde{\pi}$. From \cite{MUR}*{Theorem 6.3.5}, it follows
 that there exist unique non-degenerate representations $\tilde{\pi}_1$ and $\tilde{\pi}_2$ of $C^*_{\maxi}(\Gamma)$ and $K(\tilde{H}_X)$ on $H$ respectively, such
 that $\tilde{\pi}(a \otimes b) = \tilde{\pi}_1(a)\tilde{\pi}_2(b) = \tilde{\pi}_2(b)\tilde{\pi}_1(a)$ for all $(a,b) \in C^*_{\maxi}(\Gamma) \times K(\tilde{H}_X)$.
 The representation $\tilde{\pi}_2$ can be extended to a representation $\hat{\pi}_2$ of $L(\tilde{H}_X)$ on $H$ by \cite{DIX}*{Lemma 2.10.3} and from the same lemma
 it follows that $\tilde{\pi}_2(K(\tilde{H}_X))$ is strongly dense in $\hat{\pi}_2(L(\tilde{H}_X))$. From the double commutant theorem, it follows that the commutant of a
 $C^*$-subalgebra of $L(H)$ is strongly closed. This in turn implies that $\tilde{\pi}_1(a)\hat{\pi}_2(b) = \hat{\pi}_2(b)\tilde{\pi}_1(a)$ for
 $(a,b) \in C^*_{\max}(\Gamma) \times L(\tilde{H}_X)$. Now restrict $\tilde{\pi}_1$ to $\mathbb{C}[\Gamma]$. From \cite{MUR}*{Remark 6.3.2}, it follows that there is
 a unique $*$-representation $\hat{\pi}\colon \mathbb{C}[\Gamma] \odot L(\tilde{H}_X) \rightarrow L(H)$ with the property
 $\hat{\pi}(a \otimes b) = \tilde{\pi}_1(a)\hat{\pi}(b)$. It is clear that $\hat{\pi}$ is an extension of $\pi$.
\end{proof}

\begin{proof}[Proof of Theorem \ref{theo:abs_vers_rel}]
 The proof is analogous to that of \cite{HRY}*{Lemma 5.1}. As in Lemma~\ref{relext}, construct the algebras $C^*(\overline{B_n(Y)})^\Gamma$ by
 compressing the representation space of $C_0(X)$. The inclusions
 $\mathbb{R}(\overline{B_r(Y)})^{\Gamma} \rightarrow
 \mathbb{R}(\overline{B_R(Y)})^{\Gamma}$ for $r\le R$
 induce maps $C^*_{\maxi}(\overline{B_r(Y)})^\Gamma \rightarrow
 C^*_{\maxi}(\overline{B_R(Y)})^\Gamma$. We will show that $\varinjlim C^*_{\maxi}(\overline{B_r(Y)})^\Gamma = C^*_{\maxi}(Y \subset X)^\Gamma$. Let $A$ be a $C^*$-algebra and let 
 $\phi_r\colon C^*_{\maxi}(\overline{B_r(Y)})^\Gamma \rightarrow A$ be $C^*$-algebra morphisms such that all the diagrams of the form
 $$\begin{tikzcd}
    C^*_{\maxi}(\overline{B_r(Y)})^\Gamma \arrow{r} \arrow{d} &
    C^*_{\maxi}(\overline{B_R(Y)})^\Gamma \arrow{ld}\\ 
    A
 \end{tikzcd}$$
 with $r < R$ commute. The above compatibility condition implies the existence of a unique morphism of $*$-algebras
 $\phi: \mathbb{R}(Y\subset X)^\Gamma \rightarrow A$, such that all the diagrams 
 $$\begin{tikzcd}
    \mathbb{R}(\overline{B_r(Y)})^\Gamma \arrow{r} \arrow{d} & \mathbb{R}(Y\subset X)^\Gamma \arrow{ld}\\
    A
 \end{tikzcd}$$
 are commutative. Lemma~\ref{relext} then implies that the map $\phi$ is continuous if $\mathbb{R}(Y\subset X)^\Gamma$
 is endowed with the norm of $C^*(X)^\Gamma$. To see this note that Lemma~\ref{relext} implies that for $a \in \reals(\overline{B_r(Y)})$, $||a||_{C^*_{\max}(\overline{B_r(Y)})^\Gamma} = ||a||_{C^*_{\max}(X)^\Gamma}$. Hence, $||\phi(a)|| = ||\phi_r(a)|| \leq ||a||_{C^*_{\max}(\overline{B_r(Y)})^\Gamma} = ||a||_{C^*_{\max}(X)^\Gamma}$. Thus, $\phi$ can be extended uniquely to a morphism $C^*(Y \subset X)^\Gamma \rightarrow A$ of $C^*$-algebras.
 The universal property of the direct limit of $C^*$-algebras, implies that
 $\varinjlim C^*_{\maxi}(\overline{B_r(Y)})^\Gamma = C^*_{\maxi}(Y \subset X)^\Gamma$. The
 claim of the theorem then follows from the continuity of $K$-theory and the coarse equivalence of $\overline{B_r(Y)}$ and $\overline{B_R(Y)}$ for arbitrary
 $r,R \in \mathbb{N}$ (recall that the $K$-theory groups of the Roe algebras of coarsely equivalent spaces are isomorphic).
\end{proof}
\subsection{The Structure Algebra and Paschke Duality}
Let $X$ be as in the previous section. A representation $\rho \colon C_0(X) \rightarrow L(H)$ of $C_0(X)$ is called very ample if it is an infinite sum of
copies of an ample representation. Construct $\mathbb{R}(X)^\Gamma$ and $C^*(X)^\Gamma$ using some very ample representation. In this section we will define
a $C^*$-algebra associated to $X$ which contains $C^*_{\maxi}(X)^\Gamma$ as an
ideal and such that the $K$-theory of the quotient provides a model for
$K$-homology of $X$. 
\begin{definition}
We recall that an operator $T \in L(H)$ is called pseudolocal if it commutes with the image of $\rho$ up to compact operators; i.e., $[f,T] \in K(H)$ for all $f \in C_0(X)$. 
\end{definition}

\begin{definition}
 Denote by $\mathbb{S}_\rho(X)^{\Gamma}$ the $*$-algebra of finite propagation, pseudolocal operators in $L(H)$ which are furthermore invariant under the
 action of the group $\Gamma$. The maximal structure algebra associated to the space $X$ is the maximal $C^*$-completion of $\mathbb{S}_\rho(X)^{\Gamma}$.
 It will be denoted by $D^*_{\rho,\maxi}(X)^\Gamma$. The reduced structure algebra is the completion of the latter $*$-algebra using the norm in $L(H)$.
 We denote this algebra by $D^*_{\rho,\red}(X)^\Gamma$.
\end{definition}

\begin{remark}
From now on, we will drop $\rho$ from our notation. Later we will introduce a new completion of
$\mathbb{S}(X)^\Gamma$, which sits between the reduced and maximal completions and denote it by $D^*_q(X)^\Gamma$. If the action of $\Gamma$
is trivial, we denote the structure algebra by $D^*_{d}(X)$, where $d$ stands for the chosen completion. 
\end{remark}

  In comparison to the well known $D^*_{\red}(X)^\Gamma$, the definition and
  properties of the maximal structure algebra
  $D^*_{\maxi}(X)^\Gamma$ are trickier than one might think in the first place. First of all, one has to
  establish its existence; i.e.\,an upper bound on the $C^*$-norms. Secondly,
  we want that $C^*_{\maxi}(X)^\Gamma$ is an ideal in $D^*_{\maxi}(X)^\Gamma$ and for this
  one has to control the a priori different $C^*$-representations which are
  used in the definitions. Only then does it make sense to form
  $D^*_{\maxi}(X)/C^*_{\maxi}(X)$. Paschke duality states that its K-theory is canonically
  isomorphic to the locally finite K-homology of $X$. All of this will be done in the remainder of this
  section. 
We now introduce the so-called dual
algebras, which are larger counterparts of the Roe and structure algebra.
\begin{definition}
 Denote by $\mathfrak{C}^*(X)^{\Gamma}$ the $C^*$-algebra of $\Gamma$-invariant locally compact operators in $L(H)$. Denote by $\mathfrak{D}^*(X)^{\Gamma}$
 the $C^*$-algebra of $\Gamma$-invariant pseudolocal operators in $L(H)$. 
\end{definition}
It is clear that $\mathfrak{C}^*(X)^{\Gamma}$ is an ideal of $\mathfrak{D}^*(X)^{\Gamma}$. We have the following
\begin{theorem}
 There is an isomorphism $K_{*+1}(\frac{\mathfrak{D}^*(X)}{\mathfrak{C}^*(X)})
 \cong K^{lf}_*(X)$, where the right-hand side is the locally finite
 $K$-homology of $X$, given as the Kasparov group $KK_*(C_0(X),\mathbb{C})$.
\end{theorem}
\begin{proof}
 This is proven in \cite{PS}*{Proposition 3.4.11}.
\end{proof}

\begin{lemma} \label{calg}
 The map $\frac{\mathbb{S}(X)}{\mathbb{R}(X)} \rightarrow \frac{\mathfrak{D}^*(X)}{\mathfrak{C}^*(X)}$ induced by the
 inclusion $\mathbb{S}(X) \rightarrow \mathfrak{D}^*(X)$ is an isomorphism. In
 particular,
 $\frac{\mathbb{S}(X)}{\mathbb{R}(X)}$ is a $C^*$-algebra. The corresponding
 statement holds for the $\Gamma$-equivariant versions.
\end{lemma}
\begin{proof}
 In \cite{HR}*{Lemma 12.3.2}, the isomorphism $\frac{D^*_{\red}(X)}{C^*_{\red}(X)} \cong \frac{\mathfrak{D}^*(X)}{\mathfrak{C}^*(X)}$ is proven.
 The truncation argument used in the proof shows that $\mathfrak{D}^*(X) = \mathbb{S}(X) + \mathfrak{C}^*(X)$, which implies the surjectivity of the map
 $\frac{\mathbb{S}(X)}{\mathbb{R}(X)} \rightarrow \frac{\mathfrak{D}^*(X)}{\mathfrak{C}^*(X)}$. Injectivity is clear. An analogous argument using a suitable
 invariant open covering and partition of unity gives the isomorphism
 $\frac{\mathbb{S}(X)^{\Gamma}}{\mathbb{R}(X)^{\Gamma}} \cong \frac{\mathfrak{D}^*(X)^{\Gamma}}{\mathfrak{C}^*(X)^{\Gamma}}.$
\end{proof}

\begin{proposition}
 For $a \in \mathbb{S}(X)^\Gamma$ there exists $C_a > 0$ such that, for an arbitrary non-degenerate representation $\pi$ of $\mathbb{S}(X)^\Gamma$ we have
 $||\pi(a)|| \leq C_a$.
 \label{strbnd}
\end{proposition}
We need a few lemmas before proving Proposition~\ref{strbnd}. This proposition shows that the maximal structure algebra is well-defined. Since the structure algebra depends on both
 the coarse and topological structure of the space, the coarse geometric property of having bounded geometry alone does not guarantee the existence of the
 maximal structure algebra. This is where the properness of the metric is needed. More precisely, this is used in Lemma~\ref{calg}, which is itself used in
 the proof of Proposition~\ref{strbnd}.
\begin{lemma}
There exists a $C^*$-algebra $A \subset \mathbb{R}(X)^\Gamma$ which contains an approximate identity for
$C^*_{\maxi}(X)^\Gamma$. 
\end{lemma}
\begin{proof}
 Let $D$ be a fundamental domain for the action of $\Gamma$ on $X$. Choose a discrete subset $Y_D$ of $D$ as provided by the bounded geometry condition.
 Denote the set obtained by transporting $Y_D$ by the action of $\Gamma$ by
 $Y$. $Y$ is then clearly $\Gamma$-invariant. By \cite{EM}*{Proposition 2.7}, extended straightforwardly to the equivariant case, it suffices to
 show that there exists a $C^*$-algebra $B \subset \mathbb{R}(Y)^{\Gamma}$ which contains an approximate identity for $C^*_{\maxi}(Y)^\Gamma$. Here, as the representation space we
 choose $l^2(Y) \otimes l^2(\mathbb{N})$, where the action of $C_0(Y)$ is
 given by multiplication. By \cite{EM}*{Proposition 2.19}, $l^{\infty}(Y;C_0(\mathbb{N}))^{\Gamma} \subset \mathbb{R}(Y)^{\Gamma}$
 is a $C^*$-algebra which contains an approximate unit of $\mathbb{R}(Y)$ endowed with the reduced norm and, by Proposition~\ref{max_well_def}, of $\mathbb{R}(Y)$ endowed
 with the maximal norm. The claim then follows from density of
 $\mathbb{R}(Y)^{\Gamma}$ in $C^*(Y)^{\Gamma}$.
\end{proof}

\begin{lemma} \label{absext}
 Let $\rho$ be an arbitrary non-degenerate $*$-representation of
 $\mathbb{R}(X)^\Gamma$ on 
 some Hilbert space $H$. It  extends in a unique way to a $*$-representation
 of $\mathbb{S}(X)^\Gamma$ on $H$.

 More generally, let $\mathbb{M}(X)^\Gamma$ be the algebra of bounded
 multipliers of $\mathbb{R}(X)^\Gamma$, i.e.~all bounded operators on the
 defining Hilbert space which preserve $\mathbb{R}(X)^\Gamma$ by left and
 right multiplication. Note that $\mathbb{M}(X)^\Gamma$ contains $\mathbb{S}(X)^\Gamma$. The representation $\rho$ extends in a unique
 way to a $*$-representation of $\mathbb{M}(X)^\Gamma$.
\end{lemma}
\begin{proof}
 Let $\pi\colon \mathbb{R}(X)^{\Gamma} \rightarrow L(H)$ be a non-degenerate $*$-representation of $\mathbb{R}(X)^{\Gamma}$.
 It extends to a non-degenerate
 representation of $C^*_{\maxi}(X)^{\Gamma}$. Pick a $C^*$-subalgebra $A$ of $C^*_{\maxi}(X)^{\Gamma}$ which contains an approximate identity for
 $C^*_{\maxi}(X)^{\Gamma}$ and sits inside $\mathbb{R}(X)^{\Gamma}$. The restriction of $\pi$ to $A$ is thus also non-degenerate.
 It follows from the Cohen-Hewitt factorisation theorem (\cite{Hewitt1964}*{Theorem 2.5}) that, for all
 $w \in H$, there exist $T \in A$ and $v \in H$ with $\pi(T)v = w$. Furthermore, $\pi(S)v = 0$ for all $S \in \mathbb{R}(X)^{\Gamma}$ implies that $v$ is in the orthogonal complement of $\pi(\reals(X))H$; hence, $v=0$ by the nondegeneracy of $\pi$. It follows from \cite{FD}*{Proposition IV.3.18} that $\hat{\pi}(T)(\pi(S)v) \coloneqq \pi(TS)v$ for
 $T \in \mathbb{S}(X)^{\Gamma}$ gives a well-defined algebraic representation
 $\hat{\pi}\colon \mathbb{M}(X)^{\Gamma} \rightarrow \mathbb{L}(H)$. Here $\mathbb{L}(H)$ denotes the vector space of
 linear maps on $H$. It is clear that $\hat{\pi}$ is an extension of $\pi$. We show that $\hat{\pi}$ is actually a $*$-representation of
 $\mathbb{M}(X)^{\Gamma}$. The equalities
 $$\langle \hat{\pi}(T)(\pi(S)v), \pi(S^{\prime})v^{\prime}\rangle = \langle \pi(TS)v, \pi(S^{\prime})v^{\prime}\rangle =
 \langle \pi((S^*T^*)^*)v, \pi(S^{\prime})v^{\prime}\rangle$$
 $$= \langle v, \pi(S^*T^*S^{\prime})v^{\prime}\rangle =
 \langle \pi(S)v, \pi(T^*S^{\prime})v^{\prime}\rangle = \langle \pi(S)v, \hat{\pi}(T^*)(\pi(S^{\prime})v^{\prime})\rangle$$
 imply that the operator $\hat{\pi}(T)$ is formally self-adjoint if $T$ is self-adjoint. Furthermore, since $\hat{\pi}(T)$ is defined everywhere on $H$, it
 follows from the Hellinger-Toeplitz theorem that it is bounded. Since every element of a $*$-algebra is a linear combination of self-adjoint elements, this
 implies that the image of $\hat{\pi}$ is actually contained in $L(H)$. The previous computation then shows that $\hat{\pi}$ respects the involution; thus,
 it is a $*$-representation. Uniqueness of the extension follows from the fact that every extension $\hat{\pi}$ of $\pi$ has
 to satisfy $\hat{\pi}(T)(\pi(S)v) = \pi(TS)v$ for $T \in \mathbb{M}(X)^{\Gamma}$ and $S \in \mathbb{R}(X)^{\Gamma}$, but this determines
 $\hat{\pi}$ since all elements of $H$ are of the form $\pi(S)v$ for some $S \in \mathbb{R}(X)^{\Gamma}$ and $v \in H$.
\end{proof}

\begin{lemma}
 An arbitrary non-degenerate $*$-representation $\pi$ of $\mathbb{S}(X)^{\Gamma}$ can be decomposed as $\pi = \pi_1 \oplus \pi_2$, where both $\pi_1$ and its
 restriction to $\mathbb{R}(X)^{\Gamma}$ are non-degenerate representations on some Hilbert space $H_1$ and $\pi_2$ is a non-degenerate representation of
 $\mathbb{S}(X)^{\Gamma}$ vanishing on $\mathbb{R}(X)^{\Gamma}$.
 \label{decomp}
\end{lemma}
\begin{proof}
 This follows from Lemma~\ref{absext} and the discussion prior to \cite{AR}*{Theorem 1.3.4}.
\end{proof}

\begin{proof}[Proof of Proposition~\ref{strbnd}]
 We denote by $S$ the set of cyclic representations of $\mathbb{S}(X)^{\Gamma}$ on some Hilbert space with the property that their restriction to
 $\mathbb{R}(X)^{\Gamma}$ is a non-degenerate representation of $\mathbb{R}(X)^{\Gamma}$ on the same space. For $\pi \in S$, denote by $\pi_{\mathbb{R}}$ its
 restriction to $\mathbb{R}(X)^{\Gamma}$. The bounded geometry condition on $X$ (see Proposition \ref{prop:max_exists}) implies that $\bigoplus_{\pi \in S}\pi_{\mathbb{R}}$ is a well-defined
 non-degenerate representation of $\mathbb{R}(X)^{\Gamma}$. Lemma~\ref{absext} implies that $\Pi = \bigoplus_{\pi \in S}\pi$ is a well-defined Hilbert space
 representation of $\mathbb{S}(X)^{\Gamma}$. For $a\in \mathbb{S}(X)^{\Gamma}$
 set $C_1^a = ||\Pi(a)||$. It is shown in Lemma~\ref{calg} that
 $\frac{\mathbb{S}(X)^{\Gamma}}{\mathbb{R}(X)^{\Gamma}}$ is a $C^*$-algebra. Set $C_2^a = ||[a]||_{\frac{\mathbb{S}(X)^{\Gamma}}{\mathbb{R}(X)^{\Gamma}}}$
 and $C_a = \maxi \{C_1^a,C_2^a\}$. Now let $\pi$ be an arbitrary non-degenerate representation of $\mathbb{S}(X)^{\Gamma}$ with a decomposition
 $\pi_1 \oplus \pi_2$ as provided by Lemma~\ref{decomp}. Obviously $||\pi(a)|| \leq \maxi \{||\pi_1(a)||, ||\pi_2(a)||\}$. The claim now follows from the
 facts that $\pi_1$ is a subrepresentation of $\Pi$ and $\pi_2$ factors through $\frac{\mathbb{S}(X)^{\Gamma}}{\mathbb{R}(X)^{\Gamma}}$.
\end{proof}

\begin{proposition}
 As with the Roe algebra, the $K$-theory groups of the structure algebra are independent of the choice of the very ample representation. Furthermore, the
 assignment $X \mapsto K_*(D^*_{\maxi}(X)^\Gamma)$ is functorial for uniform
 (i.e.~coarse and continuous) maps. 
\end{proposition}
\begin{proof}
 See the discussion in \cite{HR}*{Chapter 12.4}
\end{proof} 

Lemma~\ref{absext} immediately implies the following
\begin{proposition}\label{prop:inclusion}
 $C^*_{\maxi}(X)^{\Gamma}$ is an ideal of $D^*_{\maxi}(X)^{\Gamma}$.
 \label{absideal}
\end{proposition}

\begin{proposition}
 The inclusion $\mathbb{S}(X)^{\Gamma} \rightarrow D^*_{\maxi}(X)^{\Gamma}$ gives rise to an isomorphism
 $\frac{\mathbb{S}(X)^{\Gamma}}{\mathbb{R}(X)^{\Gamma}} \cong \frac{D^*_{\maxi}(X)^{\Gamma}}{C^*_{\maxi}(X)^{\Gamma}}$. 
\end{proposition}
\begin{proof}
 Since $D^*_{\maxi}(X)^{\Gamma}$ is the maximal $C^*$-completion of $\mathbb{S}(X)^{\Gamma}$, the projection
 $\mathbb{S}(X)^{\Gamma} \rightarrow \frac{\mathbb{S}(X)^{\Gamma}}{\mathbb{R}(X)^{\Gamma}}$ gives rise
 to a morphism of $C^*$-algebras $D^*_{\maxi}(X)^{\Gamma} \rightarrow \frac{\mathbb{S}(X)^{\Gamma}}{\mathbb{R}(X)^{\Gamma}}$. Continuity of this map
 and the fact that its kernel contains $\mathbb{R}(X)^{\Gamma}$ implies that it induces a morphism
 $\frac{D^*_{\maxi}(X)^{\Gamma}}{C^*_{\maxi}(X)^{\Gamma}} \rightarrow \frac{\mathbb{S}(X)^{\Gamma}}{\mathbb{R}(X)^{\Gamma}}$. The composition
 $\frac{D^*_{\maxi}(X)^{\Gamma}}{C^*_{\maxi}(X)^{\Gamma}} \rightarrow \frac{\mathbb{S}(X)^{\Gamma}}{\mathbb{R}(X)^{\Gamma}} \rightarrow
 \frac{D^*(X)^{\Gamma}}{C^*(X)^{\Gamma}}$ is the identity on the set of classes of
 $\frac{D^*_{\maxi}(X)^{\Gamma}}{C^*_{\maxi}(X)^{\Gamma}}$ which have a representative from $\mathbb{S}(X)^{\Gamma}$.
 Since the latter set is dense, it follows that the
 composition is injective. On the other hand, by construction the composition
 $\frac{\mathbb{S}(X)^{\Gamma}}{\mathbb{R}(X)^{\Gamma}} \rightarrow \frac{D^*_{\maxi}(X)^{\Gamma}}{C^*_{\maxi}(X)^{\Gamma}} \rightarrow
 \frac{\mathbb{S}(X)^{\Gamma}}{\mathbb{R}(X)^{\Gamma}}$ is the identity  and the claim
 follows.
\end{proof}
\begin{corollary}
 There is an isomorphism $K_{*+1}(\frac{D^*_{\maxi}(X)}{C^*_{\maxi}(X)}) \cong
 K^{lf}_*(X)$.
 \label{maxpaschke}
\end{corollary}

\subsection{Yu's Localisation Algebras and $K$-homology}
\begin{definition}[\cite{QR}*{Section 2}]
  Let $A$ be a normed $*$-algebra. By $\mathfrak{T}A$ denote the normed $*$-algebra of functions $f\colon [1,\infty) \rightarrow A$ which are bounded and
  uniformly continuous.
 \end{definition}
 Clearly, if $A$ is a $C^*$-algebra, so is $\mathfrak{T}A$. Important examples for us will be the algebras $\mathfrak{T}D^*_{\maxi}(X)$ and
 $\mathfrak{T}C^*_{\maxi}(X)$
 defined using some very ample representation of $C_0(X)$.  Now we are in the position to define the localisation algebra associated to a locally compact
 metric space $X$.
 \begin{definition}[\cite{QR}*{Section 2}]
  The $C^*$-algebra generated by functions $f\in \mathfrak{T}C^*_{\maxi}(X)^\Gamma$ with the properties
  \begin{itemize}
   \item $\prop f(t) < \infty$ for all $t \in [1,\infty)$
   \item $\prop f(t) \rightarrow 0$ as $t\rightarrow \infty$ 
  \end{itemize}
  is called the localisation algebra of $X$ and is denoted by $C^*_{L,\max}(X)^\Gamma$.
  \label{localg}
 \end{definition}
 \begin{remark}
  In analogy to the fact that $C^*_{\maxi}(X)^\Gamma$ is contained as an ideal in the
  $C^*$-algebra $D^*_{\maxi}(X)^\Gamma$, one can define a $C^*$-algebra denoted by
  $D^*_{L,\maxi}(X)^\Gamma$, which contains $C^*_{L,\maxi}(X)^\Gamma$ as an ideal.
  This is the $C^*$-algebra generated by the elements in $\mathfrak{T}D^*_{\maxi}(X)^\Gamma$ with the two properties
  of Definition~\ref{localg}.
 \end{remark}

 Yu's theorem states that the $K$-theory groups of the localisation algebra
 are isomorphic to the locally finite $K$-homology groups.
 \begin{theorem}[\cite{QR}*{Theorem.\ 3.4}]
 Let $X$ be a locally compact metric space and suppose $C^*_{L,\maxi}(X)$ is defined using a very ample representation. Then the local index map
 $\ind_L\colon K^{lf}_*(X) \rightarrow K_*(C^*_{L,\maxi}(X))$ of \cite{QR}*{Definition
   2.4} is an isomorphism. Furthermore, the diagram
 $$\begin{tikzcd}
    K^{lf}_*(X) \arrow{r}{\ind_L} \arrow{rd}{\mu} & K_*(C^*_{L,\maxi}(X)) \arrow{d}{(\ev_1)_*}\\
    & K_*(C^*_{\maxi}(X))
 \end{tikzcd}$$
 is commutative. Here $\mu$ denotes the index map $K^{lf}_*(X)\cong K_{*+1}(\frac{D^*_{\maxi}(X)}{C^*_{\maxi}(X)}) \rightarrow K_*(C^*_{\maxi}(X))$.
\end{theorem}
\begin{proof}
 First note that the local index map as defined in \cite{QR} can be defined analogously in the maximal case. In \cite{QR}*{Theorem.\ 3.4} the theorem is proven for the reduced localisation algebra and uses the isomorphism
 $K_{*+1}(\frac{D_{\red}^*(X)}{C_{\red}^*(X)}) \cong K^{lf}_*(X)$. However, Corollary~\ref{maxpaschke}
 states that the isomorphism still holds if we replace the
 reduced Roe and structure algebra with the maximal ones. Thus, the argument of \cite{QR} can be used literally.
\end{proof}
Having the above theorem in mind, we will, from now on, use the notation $K^L_*(X)$ for the group $K_*(C^*_{L,\maxi}(X))$. Given a closed subset
$Y$ of $X$, we are now going to define the relative $K$-homology groups using localisation algebras and discuss the existence of a long exact
sequence for pairs. Chang, Weinberger and Yu define the relative groups by using a concrete very ample representation, which we will now describe. 

Let $Y \subset X$ be as above. Choose a countable dense set $\Gamma _X$ of $X$
such that $\Gamma_Y \coloneqq \Gamma_X \cap Y$ is dense in $Y$. Define $C^*_{L,\maxi}(X)$ and $C^*_{L,\maxi}(Y)$
using the very ample representations $H_X = l^2(\Gamma_X) \otimes l^2(\mathbb{N})$ and
$H_Y = l^2(\Gamma_Y) \otimes l^2(\mathbb{N})$ respectively. The constant family of
isometries $V_t := \iota$, where $\iota \colon H_Y \rightarrow H_X$ is the
inclusion
covers the inclusion $Y \rightarrow X$ in the sense of \cite{QR}*{Def.\ 3.1}. Hence, applying $\ad(V_t)$ pointwise we obtain a $C^*$-algebra morphism
$C^*_{L,\maxi}(Y) \rightarrow C^*_{L,\maxi}(X)$, which we will denote by $\iota(X,Y)$. Note that on elements with finite propagation
this map for each $t$ is just the extension by zero of an operator on
$H_Y$ to an operator on $H_X$. We get a map $\iota(X,Y)_*\colon K^L_*(Y) \rightarrow K^L_*(X)$.

Now denote by $K^L_*(X,Y)$ the group $K_{*-1}(C_{\iota(X,Y)})$, where $S$
denotes the suspension and $C_{\iota(X,Y)}$ the mapping cone of
$\iota(X,Y)$. The short exact sequence
$$0 \rightarrow  SC^*_{L,\maxi}(X) \rightarrow C_{\iota(X,Y)} \rightarrow C^*_{L,\maxi}(Y) \rightarrow 0$$
gives rise to a long exact sequence 
\begin{equation*}
\hdots \rightarrow  K_{*}(C^*_{L,\maxi}(Y)) \rightarrow
K_{*-1}(SC^*_{L,\maxi}(X)) \rightarrow K_{*-1}(C_{\iota(X,Y)}) \rightarrow
\cdots
\end{equation*}

of $K$-theory groups.  
Using the canonical isomorphism $K_{*-1}(S(\cdot))= K_{*}(\cdot)$ this sequence
becomes the desired long exact sequence of a pair
$$\hdots \rightarrow K^L_*(Y) \rightarrow K^L_*(X) \rightarrow K^L_*(X,Y) \rightarrow \hdots,$$
  constructed solely using localisation algebras.

\subsubsection{Relative Localisation Algebra}
Let $X$ and $Y$ be as above. We would like to extend $C^*_{L,\maxi}(Y)^{\Gamma}
\subset C^*_{L,\maxi}(X)^{\Gamma}$ to an ideal with the same $K$-theory.
\begin{definition}
 Denote by $C^*_{L,\maxi}(Y \subset X)^{\Gamma}$ the ideal in $C^*_{L,\maxi}(X)^{\Gamma}$ generated by
 functions $f \in \mathfrak{T}C^*(X)^\Gamma$ such that for all
 $t \in [1,\infty)$, $f(t)$ is supported in an $S(t)$-neighbourhood of $Y$, where $S\colon [1,\infty) \rightarrow \mathbb{R}$ is some function with
 $S(t) \rightarrow 0$ as $t\rightarrow \infty$.
\end{definition}
\begin{lemma}[\cite{RZ}*{Lemma 1.4.18}]
 Let $Y$ and $X$ be as above. The inclusion $Y \rightarrow X$ induces isomorphisms
 $K_*(C^*_{L,\maxi}(Y)^\Gamma) \cong K_*(C^*_{L,\maxi}(Y \subset X)^\Gamma)$.
\end{lemma}
\begin{proof}
 In \cite{RZ}*{Lemma 1.4.18}, this is proven in the reduced case. However
 in light of the discussion in Section \ref{sec:max_fundationas}, the
 modification of the arguments for use in  the maximal setting is
 straightforward.
\end{proof}

\subsection{Relative Group $C^*$-algebra}

Let $X$ be a proper path-connected metric space and $Y$ a path-connected subset of $X$. The inclusion $Y \rightarrow X$ induces a map
$\pi_1(Y) \rightarrow \pi_1(X)$, where we choose a point $y_0 \in Y$ to construct the fundamental groups and the latter map. This map in turn induces a
morphism $\varphi\colon C^*_{\maxi}(\pi_1(Y)) \rightarrow C^*_{\maxi}(\pi_1(X))$. The \emph{relative group $C^*$-algebra} is defined as
\begin{equation*}
 C^*_{\max}(\pi_1(X),\pi_1(Y)):= SC_\varphi.
\end{equation*}
The short exact sequence
$$0 \rightarrow  SC^*_{\maxi}(\pi_1(X)) \rightarrow C_\varphi \rightarrow C^*_{\maxi}(\pi_1(Y)) \rightarrow 0$$
and the Bott periodicity isomorphism gives a long exact sequence 
$$ \rightarrow K_*(C^*_{\maxi}(\pi_1(Y))) \rightarrow K_*(C^*_{\maxi}(\pi_1(X))) \rightarrow K_*(C^*_{\maxi}(\pi_1(X),\pi_1(Y))) \rightarrow .$$
\begin{remark}
 Note that the above $C^*$-algebras are independent of
 the chosen point $y_0$ up to an isomorphism which is well defined up to
 conjugation by a unitary and therefore is canonical on K-theory.
\end{remark}
\begin{remark}
 Recall that unless $\varphi\colon \pi_1(Y) \rightarrow \pi_1(X)$ is injective it does not necessarily induce a map of the reduced group
 $C^*$-algebras. Thus, the relative group $C^*$-algebra does not always have a reduced counterpart. 
\end{remark}

\subsection{The Relative Index Map}
The index of Chang, Weinberger and Yu is the image of a fundamental class in
$K^L_*(X,Y)$ under a mapping $\mu\colon K^L_*(X,Y) \rightarrow
K_*(C^*_{\maxi}(\pi_1(X),\pi_1(Y)))$, which they call the relative Baum-Connes
map. In this subsection we present the definition of this map along the lines of \cite{CWY}*{Section 2}. There the authors relate the $K$-theory groups of the localisation algebras and their equivariant counterparts and
exploit Theorem~\ref{eg} to relate the latter $K$-theory groups with those of the group $C^*$-algebras of the fundamental groups.

Let $X$ be a locally compact, path-connected, separable metric space and $Y$ be a closed path-connected subset of $X$. We suppose that the universal
coverings $p\colon \widetilde{X} \rightarrow X$ and $p^\prime\colon\widetilde{Y} \rightarrow Y$ of these spaces exist (e.g.\ suppose $X$ and $Y$ are $CW$-complexes) and are endowed with an invariant metric and that the metrics on $X$ and $Y$ are the
pushdowns of these metrics, i.e.\ the projections are local isometries. In the case of smooth manifolds we can start with Riemannian metrics on $X$ and $Y$ and
take their pullbacks to be the invariant Riemannian metrics on $\widetilde{X}$ and $\widetilde{Y}$. Pick countable dense subsets $\Gamma_X$ and $\Gamma_Y$ of $X$ and $Y$
such that $\Gamma_Y \subset \Gamma_X$ as before. Denote by $\Gamma_{\tilde{X}}$ and $\Gamma_{\tilde{Y}}$ the
preimages of $\Gamma_X$ and $\Gamma_Y$, respectively. Construct the (equivariant) Roe algebras and the (equivariant) localisation algebras using the representations
$l^2(\Gamma_\cdot) \otimes l^2(\mathbb{N})$. We recall that the equivariant algebras are
constructed using the action of fundamental groups by deck transformations.

\begin{proposition}[\cite{CWY}*{Proposition 2.8}] \label{eps}
 Let $X$ and $\tilde{X}$ be as above. Suppose furthermore that $X$ is compact. Then there exists an $\epsilon > 0$ depending on $X$ such that for finite
 propagation locally compact operators $T$ with $\prop(T) < \epsilon$, the
 kernel $\tilde{k}$ defined in the following defines an element of
 $C^*_{\maxi}(\tilde{X})^{\pi_1(X)}$, which we will denote by $L(T)$.

  Observe for the definition of $L(T)$ that a finite propagation locally
  compact operator $T$ on
  $l^2(\Gamma_X) \otimes H$ with $\prop(T) = r$ is given by a matrix
  $\Gamma_X \times \Gamma_X \xrightarrow{k} K(H)$ such that $k(x,x^\prime)$ is
  $0$ for all $(x,x^\prime) \in \Gamma_X \times \Gamma_X$ with
  $d_X(x,x^\prime) \geq r$. Define the lifted operator on
  $l^2(\Gamma_{\tilde{X}}) \otimes H$ using the matrix
  $(\tilde{x},\tilde{x}^\prime) \xmapsto{\tilde{k}}
  k(p(\tilde{x}),p(\tilde{x}^\prime))$ if
  $d_{\tilde{X}}(\tilde{x},\tilde{x}^\prime) < r$ and $0$ otherwise.

  Vice versa, every equivariant kernel $\tilde T\in C^*_{\maxi}(\tilde X)^{\pi_1(X)}$ of propagation $<\epsilon$
  is such a lift, and this in a unique way, defining the \emph{push-down}
  $\pi(\tilde T)\in C^*_{\maxi}(X)$ as the inverse of the lift.

 For the appropriate choice of $\epsilon$, the covering $\tilde X\to X$ should
 be trivial when restricted to balls say of radius $2\epsilon$.
\end{proposition}

\begin{remark}\label{rem:ext_of_prop_lift}
 Later we will need a slight generalisation of Proposition~\ref{eps} for manifolds obtained by attaching an infinite cylinder to a compact manifold with
 boundary. It is evident that the $\epsilon$ obtained for the manifold with
 boundary also works for the manifold with the infinite cylinder attached, and
 then the construction indeed goes through without any modification.
\end{remark}

\begin{definition}\label{def:loc_lift}
  Let $T\colon s\mapsto T_s$ be an element of $\mathbb{R}_L(X)$, i.e.\ $T_s$ is
  locally compact and has finite propagation which tends to $0$ as
  $s\to\infty$. Therefore $\prop(T_{s}) < \epsilon$ for all $s \geq s_T$ with
  some $s_T\in [1,\infty)$. Define the lift
  \begin{equation*}
L(T)\colon s\mapsto
\begin{cases}
  L(T_{s_T}); & s\le s_T\\
  L(T_s); & s\ge s_T
\end{cases}
 \end{equation*}
 to obtain an element in
 $C^*_{L,\maxi}(\tilde{X})^{\pi_1(X)}$.

 Similarly, for $\tilde T\colon s\mapsto \tilde T_s$ an element of
 $\mathbb{R}_L(\tilde X)^{\pi_1(X)}$ such that $\tilde T_s$ is locally
 compact, equivariant and has finite propagation which tends to $0$ as
 $t\to\infty$ (in particular $\prop(\tilde T_{s}) < \epsilon$ for all $s \geq
 s_{\tilde T}$ for some
  $s_{\tilde T}\in [1,\infty)$) define its \emph{push-down}
  \begin{equation*}
    \pi(\tilde T)\colon s\mapsto
    \begin{cases}
      \pi(\tilde T_{s_{\tilde T}}); & s\le s_{\tilde T}\\
      \pi(\tilde T_s); & s\ge s_{\tilde T}.
    \end{cases}
  \end{equation*}
\end{definition}

\begin{proposition}\label{prop:lift_iso}
  Set $C_0^*(\widetilde X)^{\pi_1(X)}:= C_0([1,\infty), C^*_{\maxi}(\tilde
    X)^{\pi_1(X)})$, the ideal of $C^*_{L,\max}(\widetilde{X})^\Gamma$ consisting of functions whose norm tends to $0$ as
    $s\to\infty$. 
  The assignments of Definition \ref{def:loc_lift} give rise to continuous
  $*$-homomorphisms
  \begin{equation*}
    \begin{split}
      L &\colon \mathbb{R}_L(X) \to {C^*_{L,\maxi}(\tilde
        X)^{\pi_1(X)}}/{C^*_0(\tilde X)^{\pi_1(X)}}\\
      \pi & \colon \mathbb{R}_L(\tilde X)^{\pi_1(X)} \to
      C^*_{L,\maxi}(X)/{C^*_0(X)},
    \end{split}
  \end{equation*}
  where we use that the algebra of functions vanishing at $\infty$ is an ideal
  of the localisation algebra. Being continuous, they extend to the
  $C^*$-completions, and they evidently map the ideal $C_0([1,\infty),C^*_{\maxi}(X))$
  or $C_0([1,\infty),C^*_{\maxi}(\tilde X)^{\pi_1(X)})$ to $0$, so that we get
  $C^*$-algebra homomorphisms
  \begin{equation*}
    \begin{split}
      L &\colon {C^*_{L,\maxi}(X)}/{C_0^*(X)}\to {C^*_{L,\maxi}(\tilde
        X)^{\pi_1(X)}}/{C_0^*(\tilde X)^{\pi_1(X)}}\\
      \pi & \colon {C^*_{L,\maxi}(\tilde
        X)^{\pi_1(X)}}/{C_0^*(\tilde X)^{\pi_1(X)}} \to
      {C^*_{L,\maxi}(X)}/{C_0^*(X)}.
     \end{split}
  \end{equation*}
  By construction these two homomorphisms are inverse to
  each other.

  Being cones,
  $C_0([1,\infty),C^*_{\maxi}(\tilde X)^{\pi_1(X)})$ and $C_0(1,\infty),C^*_{\maxi}(X))$ have vanishing K-theory and
 by the 6-term exact sequence the projections induce isomorphisms in K-theory
  \begin{equation*}
    \begin{split}
      K_*(C^*_{L,\maxi}(\tilde X)^{\pi_1(X)}) &\to K_*(C^*_{L,\maxi}(\tilde
      X)^{\pi_1(X)}/C_0^*(\tilde X)^{\pi_1(X)}),\\
      K_*(C^*_{L,\maxi}(X))& \to K_*(C^*_{L,\maxi}( X)/C_0^*(X)).
    \end{split}
  \end{equation*}
  We therefore get a well defined induced isomorphism in K-theory
  \begin{equation*}
    L_*\colon K_*^L(X)= K_*(C_{L,\maxi}^*(X))\to K_*(C_{L,\maxi}^*(\tilde X)^{\pi_1(X)})
  \end{equation*}
  with inverse $\pi_*$.
\end{proposition}

The proof of Proposition \ref{prop:lift_iso} is not trivial, as we have to
come to grips with the potentially different representations which enter the
definition of the maximal $C^*$-norms for $C^*_{\maxi}(X)$ and $C^*_{\maxi}(\tilde
X)^{\pi_1(X)}$. To do this, we use the following lemma. 

\begin{lemma} \label{decp}
 Let $\epsilon$ be as in Proposition~\ref{eps}. There exists $K \in \mathbb{N}$, such that for all $T \in \mathbb{R}(X)$ and
 $\tilde T\in \mathbb{R}(\tilde X)^{\pi_1(X)}$ with propagation less than $\epsilon$
 we have $||L(T)||_{C^*_{\maxi}(\tilde X)^{\pi_1(X)}} \leq K ||T||_{C^*_{\maxi}(X)}$ and $||\pi(\tilde T)||_{C^*_{\maxi}}(X)\le
 K ||\tilde T||_{C^*_{\maxi}(\tilde X)^{\pi_1(X)}}$.
\end{lemma}
\begin{proof}
 By assumption, $X$ has bounded geometry. Consequently, we can and do choose
 for some fixed $c>0$ a $c$-dense uniformly discrete subset $D$ of $\Gamma_X$
 and denote by
 $C^*_{\maxi}(D)$ and $C^*_{\maxi}(\tilde D)^{\pi_1(X)}$ the Roe algebras of $D$
 constructed using $l^2(D) \otimes H$ and $l^2(\tilde D) \otimes H$ as
 before. The proof of
 \cite{GongWangYu}*{Lemma 3.4} guarantees the existence of a $K\in \mathbb{N}$ 
 such that for all $T \in C^*_{\maxi}(D)$ with $\prop(T) < \epsilon$ there
 exist operators $T_{i\in\{1,\hdots,K\}} \in C^*_{\maxi}(D)$ such that $||T_i|| \leq ||T||$, 
 $T_i^*T_i \in l^\infty(D;K(H))$, i.e.~$T_i^*T_i$ are
 operators of propagation $0$, and such that $\sum{T_i} = T$. Moreover, the
 lift $\tilde T_i$ satisfies that $\tilde T_i^*\tilde T_i=\widetilde{T_i^*T_i}
 \in l^\infty(\tilde D; K(H))^{\pi_1(X)}\stackrel{L}\cong
 l^\infty(D;K(H))$. Hence the norm of
 $\tilde T_i^*\tilde T_i$ is exactly $||T_i||^2$.

 We thus have
 $||L(T)|| \leq K||T||$. With a completely analogous argument we get
 $||\pi(\tilde T)||\le K|| \tilde T||$.

 Note that there are isomorphisms
 $$C^*_{\maxi}(X) \rightarrow C^*_{\maxi}(D),$$
 $$C^*_{\maxi}(\tilde{X})^{\pi_1(X)} \rightarrow C^*_{\maxi}(\tilde D)^{\pi_1(X)}$$
 which can be constructed explicitly (compare \cite{GongWangYu}*{Section
   4.4}). These isomorphisms can be chosen so as to make the diagrams
 $$\begin{tikzcd}
    \mathbb{R}(\tilde{X})^{\pi_1(X)}_{\epsilon} \arrow{r}  &
    \mathbb{R}(\tilde D)^{\pi_1(X)}_{\epsilon}  \\
    \mathbb{R}(X)_\epsilon \arrow{u}{L} \arrow{r}  & \mathbb{R}(D)_{\epsilon} \arrow{u}{L}
  \end{tikzcd}\qquad
\begin{tikzcd}
    \mathbb{R}(\tilde{X})^{\pi_1(X)}_{\epsilon} \arrow{r}\arrow{d}{\pi}  &
    \mathbb{R}(\tilde D)^{\pi_1(X)}_{\epsilon}\arrow{d}{\pi}  \\
    \mathbb{R}(X)_\epsilon \arrow{r}  & \mathbb{R}(D)_{\epsilon} 
  \end{tikzcd}
  $$
 commute. Here the subscript $\epsilon$ means that we are only considering
 operators with propagation less than $\epsilon$. 

 The latter commutative diagrams complete the proof.
\end{proof}

\begin{proof}[Proof of Proposition \ref{prop:lift_iso}] 
  Recall that for $(\tilde T\colon s\to \tilde T_s)\in C_{L,\maxi}^*(\tilde X)^{\pi_1(X)}$ we use
  the supremum norm: $||\tilde T|| = \sup_{s\in [1,\infty)}||\tilde T_s||$. It
  follows that the norm of the image of $\tilde T$ in $C_{L,\maxi}^*(\tilde X
  )^{\pi_1(X)}/C_0([1,\infty);C^*(\tilde X)^{\pi_1(X)})$ under the projection map is $||[\tilde T]|| =
  \limsup_{s\in[1,\infty)} ||\tilde T_s||$ (specifically, multiplication of
  $\tilde T$ with a cutoff function $\rho\colon [1,\infty)\to [0,1]$ which
  vanishes on $[1,R]$ and is identically $1$ on $[R+1,\infty)$ produces
  representative of $[\tilde T]$ whose norm in $C^*_L(\tilde X)^{\pi_1(X)}$
  approaches $\limsup_{s\in [1,\infty)} ||\tilde T_s||$ as $R\to\infty$).

  The assertion then follows immediately from Lemma \ref{decp}.
\end{proof}

Until the end of Section 2.5 we are going to suppose that $X$ is compact and
that $Y$ is a closed subset of $X$. 
Recall that $\varphi$ denotes the map $\pi_1(Y) \rightarrow \pi_1(X)$ induced by
the inclusion. Following the notation introduced in \cite{CWY}*{Section 2}, we
denote 
by $Y^\prime$ the set $p^{-1}(Y)$ and by $p^{\prime \prime}\colon Y^{\prime \prime} \rightarrow Y$ the covering of $Y$ associated to the
subgroup $\kernel \varphi$; hence, $Y^\prime = \pi_1(X) \times_{\pi_1(Y)/\kernel \varphi}Y^{\prime \prime}$. Now construct the equivariant
Roe and localisation algebras for
$Y^\prime$ and $Y^{\prime \prime}$ using the sets
$p^{-1}(\Gamma_Y)$ and $(p^{\prime \prime})^{-1}(\Gamma_Y)$ similarly as before. 
\begin{theorem}[\cite{CWY}*{Lemma 2.12}]
 There is a map
 $$\psi^{\prime \prime}\colon C^*_{\maxi}(\tilde{Y})^{\pi_1(Y)} \rightarrow C^*_{\maxi}(Y^{\prime \prime})^{\pi_1(Y)/\kernel \varphi}$$
 with the property that there exists $\epsilon > 0$ such that given an operator $T \in C^*_{\maxi}(\tilde{Y})^{\pi_1(Y)}$ with
 $\prop(T) < \epsilon$ and kernel
 $k$ on $(p^\prime)^{-1}(\Gamma_Y)$ the pushdown of $k$ gives a unique well-defined kernel $k_Y$ on $\Gamma_Y$ and $\psi^{\prime \prime}(T)$ is given by
 the kernel $(x,y) \mapsto k_Y(p^{\prime \prime}(x),p^{\prime \prime}(y))$ for $x,y \in Y^{\prime \prime}$ with $d_{Y^{\prime \prime}}(x,y) < \epsilon$.
 \label{difffund}
\end{theorem}
\begin{remark}
 It can be observed from the proof of Theorem \ref{difffund}, that the result can be generalised to obtain a map $C^*_{\maxi}(Z)^\Gamma \rightarrow C^*_{\maxi}(Z/N)^{\Gamma/N}$, where $Z$ is a bounded geometry space satisfying the properties mentioned in the beginning of the paper, $\Gamma$ is a discrete group acting freely and properly on $Z$ via isometries, $N\subset \Gamma$ is a normal subgroup and there exists an $\epsilon$ such that the coverings $Z \rightarrow Z/N^\prime$ are trivial when restricted to $\epsilon$-balls for any normal subgroup $N^{\prime} \subset \Gamma$.
\end{remark}

\begin{remark}
 For the proof of Theorem~\ref{difffund}, Chang, Weinberger and Yu use that
 the push-down of operators with small propagation as defined in Definition \ref{def:loc_lift} can be extended to an honest $*$-homomorphism. Doing it partially
 gives a morphism of $*$-algebras 
 $\psi^{\prime \prime}\colon \mathbb{R}(\tilde{Y})^{\pi_1(Y)} \rightarrow \mathbb{R}(Y^{\prime \prime})^{\pi_1(Y)/\kernel \varphi}$ and then
 maximality of the norms
 provides the extension to the desired $C^*$-homomorphism $C^*_{\maxi}(\tilde
 Y)^{\pi_1(Y)}\to C^*_{\maxi}(Y'')^{\pi_1(Y)/\ker \varphi}$. Note that, in general, this is not possible
 if we use the reduced equivariant Roe algebras.
\end{remark}

Using $Y^\prime = Y^{\prime \prime} \times_{\pi_1(Y)/\kernel \varphi} \pi_1(X)$, we get a $C^*$-algebra morphism
\begin{equation*}
\psi^\prime\colon C^*_{\maxi}(Y^{\prime \prime})^{\pi_1(Y)/\kernel \varphi} \rightarrow
C^*_{\maxi}(Y^\prime)^{\pi_1(X)} \subset C^*_{\maxi}(\tilde{X})^{\pi_1(X)}
\end{equation*}
where the first map repeats the operators on the different copies of $Y''$
inside $Y'$.
Composing $\psi^\prime$ and $\psi^{\prime \prime}$ we obtain the map
$\psi\colon C^*_{\maxi}(\tilde{Y})^{\pi_1(Y)} \rightarrow
C^*_{\maxi}(\tilde{X})^{\pi_1(X)}$. Application of the maps pointwise defines the
corresponding maps for localisation algebras, which we denote with the same
symbols with subscript $L$. 

\begin{theorem}\label{prop:Lift_compatible_with_Y_to_X}
  The constructions just described fit into the following commutative diagram
  of $C^*$-algebras, where the composition in the third row is the map
  $\psi_L$, in the forth row is $\psi$, and in the last row is $\varphi$. The
  projection maps in the second row of vertical maps are K-theory isomorphism.
  The last vertical maps induce the canonical isomorphism in K-theory of
  Theorem \ref{eg}. The Roe and localisation algebras are constructed using
  the maximal completion (we write $\Max$ instead of $\max$ for typesetting reasons).
{\tiny  \begin{equation*}
    \begin{CD}
      C^*_{L,\Max}(Y) @= C^*_{L,\Max}(Y) @= C^*_{L,\Max}(Y) @>\iota>> C^*_{L,\Max}(X)\\
      @VV{L}V @VV{L}V @VV{L}V @VV{L}V\\
    \frac{C^*_{L,\Max}(\tilde Y)^{\pi_1(Y)}}{C_0^*(\tilde Y)^{\pi_1(Y)}}
      @>{\psi''_L}>>    \frac{C^*_{L,\Max}( Y'')^{\pi_1(Y)/\ker \varphi}}{C_0^*(Y'')^{\pi_1(Y)/\ker\varphi}}  @>{\psi'_L}>>    \frac{C^*_{L,\Max}(
          Y')^{\pi_1(x)}}{C_0^*(Y')^{\pi_1(x)}} @>{\subset}>>
          \frac{C^*_{L,\Max}(\tilde X)^{\pi_1(X)}}{C_0^*(\tilde
            X)^{\pi_1(X)}} \\
            @AAA @AAA @AAA @AAA\\
         C^*_{L,\Max}(\tilde Y)^{\pi_1(Y)} @>{\psi''_L}>>
         C^*_{L,\Max}(Y'')^{\pi_1(Y)/\ker\varphi} @>{\psi'_L}>> C^*_{L,\Max}(Y')^{\pi_1(X)}
         @>{\subset}>> C^*_{L,\Max}(\tilde X)^{\pi_1(X)} \\
         @VV{\ev_1}V        @VV{\ev_1}V       @VV{\ev_1}V       @VV{\ev_1}V\\
         C^*_{\Max}(\tilde Y)^{\pi_1(Y)} @>{\psi''}>>
         C^*_{\Max}(Y'')^{\pi_1(Y)/\ker\varphi} @>{\psi'}>> C^*_{\Max}(Y')^{\pi_1(X)}
         @>{\subset}>> C^*_{\Max}(\tilde X)^{\pi_1(X)}\\
         @AAA @AAA @AAA @AAA\\
         C_{\Max}^*(\pi_1(Y)) @>{pr_*}>> C_{\Max}^*(\pi_1(Y)/\ker\varphi) @>>>
         C_{\Max}^*(\pi_1(X)) @= C_{\Max}^*(\pi_1(X)).
    \end{CD}
  \end{equation*}
}
\end{theorem}
\begin{proof}
  If in the first row $C^*_L(Y)$ is replaced by $\reals_L(Y)$, then the definition of $L$, the behaviour of the push-down map $\psi''$ and the
  (trivial) lifting map $\psi'$ on operators of small propagation and the definition of $\iota$ and $\subset$ imply the commutativity of the first two rows of the diagram. The continuity of the involved maps then implies the commutativity of the first two rows. In order to show the commutativity of the last two rows we recall the isomorphisms
  $K_*(C^*_{\max}(\pi_1(\cdot))) \rightarrow
  K_*(C^*_{\maxi}(\tilde{\cdot})^{\pi_1(\cdot)})$. For this we need the isomorphisms 
  $C^*_{\max}(\pi_1(\cdot)) \otimes K(H)\xrightarrow{\iso}C^*_{\maxi}(\tilde{\cdot})^{\pi_1(\cdot)}$. Here we modify the proof of
 \cite{HR}*{Lemma 12.5.3} slightly to suit our choice of the representation space. Choose a countable dense subset $D$ of the fundamental domain of
 $\tilde{Y}$ such that $D$ and $gD$ are disjoint for $g \neq e$ in $\pi_1(Y)$. With $\Gamma_{\tilde{Y}} = \bigsqcup_{g\in \pi_1(Y)}gD$, we get an
 isomorphism $l^2(\Gamma_{\tilde{Y}}) \otimes l^2(\mathbb{N}) \cong l^2(\pi_1(Y)) \otimes (\underset{n \in \mathbb{N}}{\oplus}l^2(D)).$ Using this
 isomorphism we then obtain a $*$-isomorphism between $\mathbb{C}[\pi_1(Y)] \otimes K(\underset{n \in \mathbb{N}}{\oplus}l^2(D))$ and the algebra of
 invariant, finite propagation and locally compact operators. This induces the
 desired isomorphism
$C^*_{\max}(\pi_1(Y)) \otimes K(\underset{n \in \mathbb{N}}{\oplus}l^2(D)) \xrightarrow{\iso} C^*_{\maxi}(\tilde{Y})^{\pi_1(Y)}$. Furthermore we note (see
\cite{RLL}*{Proposition 6.4.1 and Proposition 8.2.8}) that the standard
isomorphisms $K_p(A) \rightarrow K_p(A \otimes K(H))$ for a $C^*$-algebra $A$
and a separable 
infinite dimensional Hilbert space $H$ is induced by the morphism $a \mapsto a
\otimes p$, with $p$ a rank one projection. Now consider the
rank one projection $p_{x_0} \otimes p_1$ on $\underset{n \in \mathbb{N}}{\oplus}l^2(D) \cong l^2(D) \otimes l^2(\mathbb{N})$ for some $x_0 \in D$
and $p_1$ the operator on $l^2(\mathbb{N})$ projecting to the first
component. The composition gives the desired map $C^*_{\maxi}(\pi_1(Y))
\rightarrow C^*_{\maxi}(\tilde{Y})^{\pi_1(Y)}$ which induces the $K$-theory
isomorphism of Theorem~\ref{eg}. We can perform the same procedure for
$Y^{\prime \prime} = \tilde{Y}/(\kernel \varphi)$. Considering the above $D$ (or
rather its image under $\tilde Y\to Y''$) as a subset of $Y^{\prime \prime}$
and using
$\Gamma_{Y^{\prime \prime}} = \bigsqcup_{g\in \frac{\pi_1(Y)}{\kernel
    \varphi}}gD$, we get the corresponding isomorphism
$l^2(\Gamma_{Y^{\prime \prime}}) \otimes l^2(\mathbb{N}) \cong l^2(\frac{\pi_1(Y)}{\kernel \varphi})
\otimes (\underset{n \in \mathbb{N}}{\oplus}l^2(D)).$
Choosing the same $p$ as above our procedure defines the desired
$C^*_{\maxi}(\frac{\pi_1(Y)}{\kernel \varphi}) \rightarrow C^*_{\maxi}(Y^{\prime
  \prime})^{\frac{\pi_1(Y)}{\kernel \varphi}}$ which is a K-theory isomorphism
and which makes the lower left corner of the diagram of Theorem
\ref{prop:Lift_compatible_with_Y_to_X} commutative.
Similarly we construct the corresponding map for $Y^{\prime}$, which is the associated bundle to $Y^{\prime \prime}$ with
fibre $\pi_1(X)$ (we can consider the above $D$ as a subset of
$Y^{\prime}$). The construction gives rise to the morphism
$C^*_{\maxi}(\pi_1(X))\to C^*_{\maxi}(Y')^{\pi_1(X)}$ which is a K-theory isomorphism
and which makes the lower middle square of the diagram of Theorem
\ref{prop:Lift_compatible_with_Y_to_X} commutative. 
Finally, considering $D$ as a subset of $Y^{\prime}$ and extending it to a dense subset
of a fundamental domain of $\tilde{X}$, we obtain, similarly as above,
 a corresponding map for $\tilde{X}$, the morphism $C^*_{\max}(\pi_1(X))\to
 C^*_{\maxi}(\tilde X)^{\pi_1(X)}$ which is a K-theory isomorphism such that also the
 lower  right corner of the diagram of Theorem
 \ref{prop:Lift_compatible_with_Y_to_X} commutes. This finishes the proof of the
 said Theorem.
\end{proof}

\begin{definition}\label{def:relative_index}
  The commutative diagram of Theorem
  \ref{prop:Lift_compatible_with_Y_to_X} defines a zig-zag of maps between the
  mapping cones of the compositions of the maps from left to right. Using in
  addition that the two wrong way vertical maps induce isomorphisms
  in K-theory, we obtain the map
  \begin{equation*}
  \mu \colon  K_*( SC_{\iota(X,Y)}) \to
  K_*(SC_{\varphi})\stackrel{\text{Def}}{=} K_*(C^*_{\max}(\pi_1(X),\pi_1(Y))),
  \end{equation*}
  which we call the \emph{relative index map}. In \cite{CWY} it is called the
  \emph{maximal relative Baum-Connes map}.
\end{definition}
\section{A Geometric and Functorial Completion of the Equivariant Roe Algebra}
\label{sec:quotient_completion}

\subsection{Maximal Roe Algebra and Functions of the Dirac Operator}

\label{sec:max_Roe_Dirac}
Before
describing our geometric completion of the algebraic Roe algebra, we discuss
issues arising in coarse index theory when one uses maximal completions of the
relevant $C^*$-algebras, which lead to gaps in \cite{CWY}. 
A crucial role in coarse index theory is played by functions of the Dirac
operator (via functional calculus). If we work with the usual (reduced) Roe
algebras, the latter are defined as algebras of bounded operators on
$L^2$-spinors, and the Dirac operator is an unbounded operator on the same
Hilbert space. Ellipticity and finite propagation of the wave operator then
are used to show that certain functions of the Dirac
operator satisfy the defining conditions for the reduced Roe
algebra and of the reduced structure algebra.

However, if one uses the maximal versions this is highly non-trivial:
\begin{enumerate}
\item The functions $f(D)$ which do have finite propagation are by the
  very definition elements of the algebraic Roe algebra (if $f$ vanishes at
  infinity) or of the algebraic structure algebra (if $f$ is a normalising
  function). The wave operators $e^{itD}$ are bounded multipliers of the
  maximal Roe 
  algebra and by Lemma \ref{absext} act as bounded operators on the defining
  representation of the 
  maximal Roe algebra.
\item However, it is not obvious at all that the one parameter group
  $t\mapsto e^{itD}$ is strongly continuous on any Hilbert space on which the maximal Roe algebra is represented faithfully, i.e.~is obtained from an (unbounded) self-adjoint
  operator $D$ on such a Hilbert space. Thus one needs to have a reasonable
  definition of $f(D)$ in the maximal Roe and structure algebra for $f$
  without a compactly supported Fourier transform.
\item Even if one manages to construct the self-adjoint unbounded operator $D$
  on the maximal representation, it remains to show that this maximal
  Dirac operator is invertible if the underlying manifold has uniformly
  positive scalar curvature: one has to make sense also of a (geometric)
  Schr\"odinger-Lichnerowicz formula for this non-geometric representation?
\end{enumerate}

Chang, Weinberger, and Yu's article \cite{CWY} takes all these necessary
constructions and properties 
for granted, without any justification. We propose a way around by passing to
a slightly different and much more convenient completion. Later, Guo, Xie, and
Yu posted the preprint \cite{GuoXieYu} where they also identify these gaps in
\cite{CWY} and  propose positive answers to the above questions.

\subsection{The Quotient Completion}
Our suggestion to overcome the problems addressed in Section
\ref{sec:max_Roe_Dirac} is to work with another functorial completion of the
equivariant Roe 
algebra which is more geometric.
We are studying the case that a group $\Gamma$ acts freely and
properly discontinuously by isometries on a proper metric space $X$.

For every normal subgroup $N\subset \Gamma$ we then can form the metric space
$X/N$ on which the quotient group $Q:=\Gamma/N$ acts as before. Indeed,
typically we obtain $X$ as a $\Gamma$-covering of a space $X/\Gamma$ and the
$X/N$ are then other normal coverings of $X/\Gamma$.

In the usual way, the purely algebraically defined algebras $\mathbb{R}(X)^\Gamma$ and
$\mathbb{S}(X)^\Gamma$ act via their images in $\mathbb{R}(X/N)^{\Gamma/N}$
and $\mathbb{S}(X/N)^{\Gamma/N}$ on all these quotients (see Theorem \ref{difffund}), and we complete with
respect to all these norms at once. Denote the corresponding completions by
$C^*_q(X)^\Gamma$ and $D^*_q(X)^\Gamma$. It is clear that the former is an
ideal in the latter. It is also clear that this has the usual functoriality
properties for $\Gamma$-equivariant maps for fixed $\Gamma$, but now in
addition is functorial (this is built in) for the quotient maps $X\to X/N$,
giving $C_q^*(X)^\Gamma\to C_q^*(X/N)^{\Gamma/N}$ and $D_q^*(X)^\Gamma\to
D_q^*(X/N)^{\Gamma/N}$.

Finally, for inclusion of groups $\iota\colon \Gamma\to G$ 
induces an \emph{induction map} $C_q^*(X)^\Gamma\to C_q^*(X\times_\Gamma G)^G$,
because for every quotient $G/N$ we get the associated induction
$$\mathbb{R}(X/{\Gamma\cap N})^{\Gamma/(\Gamma\cap N)} \to \mathbb{R}(X/(\Gamma\cap
N)\times_{\Gamma/\Gamma\cap N}G/N)^{G/N}=\mathbb{R}(X\times_\Gamma
G/N)^{G/N}.$$
The corresponding construction works for $D^*_q$ and for the localisation
algebras.

Putting this together, we get the expected functoriality of $C^*_q$ and
$D^*_q$ and the localisation algebras for maps equivariant for any homomorphism $\alpha\colon\Gamma_1\to \Gamma_2$.

\begin{lemma}
  Suppose $\Gamma$ acts cocompactly on $X$. Then $C^*_q(X)^\Gamma$ is isomorphic to
  $C^*_q(\Gamma)\otimes K(H)$. Here, $C^*_q(\Gamma)$ is the $C^*$-completion
  of $\mathbb{C}[\Gamma]$ in the representation 
  $\bigoplus_{N\normalsubgroup\Gamma}l^2(\Gamma/N)$, where the sum is over all normal
  subgroups $N$ of $\Gamma$.
\end{lemma}
\begin{proof}
  The proof is precisely along the lines of the one of Theorem \ref{eg}. 
\end{proof}

\begin{proposition}
  Let $X/\Gamma$ be a complete Riemannian spin manifold with $\Gamma$-covering
  $X$. The Dirac operator on the different normal coverings $X/N$ for the
  normal subgroups $N$ of $\Gamma$ gives rise to a self-adjoint unbounded
  operator in the defining representation of $C^*_q(X)^\Gamma$. If $f\in
  C_0(\reals)$ we get $f(D)\in C^*_q(X)^\Gamma$, if $\Psi\colon
  \reals\to [-1,1]$ is a normalising function, we get $\Psi(D)\in
  D^*_q(X)^\Gamma$.

  This construction is functorial for the quotient maps $X\to X/N$ for
  normal subgroups $N\normalsubgroup\Gamma$.

  The Schr\"odinger-Lichnerowicz argument applies: if $X/\Gamma$ has uniformly
  positive scalar curvature then the spectrum of the operator $D$ in the
  defining representation of $C^*_q(X)^\Gamma$ does not contain $0$.

  Let $A\subset X$ be a $\Gamma$-invariant measurable subset. Then $\chi_A$,
  the operator of multiplication with the characteristic function of $A$ is an
  element of $D^*_q(X)^\Gamma$, in particular a multiplier of
  $C^*_q(X)^\Gamma$. Under the quotient map $X\to X/N$ for a normal subgroup
  $N\normalsubgroup \Gamma$ it is mapped to $\chi_{A/N}$. Similarly, a function
  of the Dirac operator on $X$ is mapped to the same function of the Dirac
  operator on $X/N$.
\end{proposition}
\begin{proof}
  The statements about the Dirac operator are just an application of the usual
  arguments to all the quotients $X/N$ simultaneously, using Lemma \ref{relind}.

  The statement about $\chi_A$ is a direct consequence of the definitions.
\end{proof}

\begin{remark}
 We note that all the statements in Section~\ref{sec:maxroe} have a counterpart when we use the quotient completion instead of the maximal
 completion of the equivariant algebras and their proofs are completely
 analogous to (and often easier than) the proofs for the maximal completions. In particular, we
 have a relative index map in this case. Furthermore we would like to emphasise that Theorem~\ref{theo:abs_vers_rel} holds for the quotient completion.
 Given a map
 $\phi\colon \Gamma \rightarrow \pi$, we get by functoriality a morphism $\phi: C_q^*(\Gamma) \rightarrow C_q^*(\pi)$, and $C_q^*(\pi,\Gamma)$ will denote
 $SC_\phi$ as before.
\end{remark}

\section{Higher Indices of Dirac Operators on Manifolds with Boundary} \label{relindd}

\subsection{Construction of the Relative Index}

Throughout this section, we consider only even dimensional spin manifolds.
We define the relative index of the Dirac operator of a
manifold $M$ with boundary $N$ in the following groups: 
\begin{itemize}
 \item in $C^*_{\maxi}(\pi_1(M),\pi_1(N))$,
 \item in $C^*_q(\pi_1(M),\pi_1(N))$ and
 \item in $C^*_{\red}(\pi_1(M),\pi_1(N))$ if $\pi_1(N)\to \pi_1(M)$ is injective.
\end{itemize}
In what follows the subscript $d$ stands for one of the mentioned
completions.
Before defining the relative index of the Dirac operator on a manifold with
boundary, we recall the explicit image of the fundamental class under the local
index map.
Given a complete Riemannian spin manifold $X$ with a free and proper action
of $\Gamma$ by isometries, denote by $D_{X}$ the Dirac operator on $X$. Let $\Psi_t$ be a 
sup-norm continuous family of normalising functions, i.e.\ each $\Psi_t$ is an
odd, smooth
function $\Psi_t\colon \mathbb{R}\to [-1,1]$ such that $\Psi_t(s)
\xrightarrow{s\to\infty} 1$. Suppose furthermore that for $t\ge 1$ the distributional Fourier transform of $\Psi_t$ is supported in a $\frac{1}{t}$-neighbourhood of
$0$.
Choose an isometry $\alpha$ between
$L^2(\slashed{S}^+)$ and 
$L^2(\slashed{S}^-)$ induced from a measurable bundle isometry, set
$\Psi_t(D_X)^+:=\Psi_t(D_X)|_{L^2(\slashed{S}^+)}$ and $F_X(t):=\alpha^*\circ
\Psi_t(D_X)^+$. 
Set $e_{11}:=
\begin{pmatrix}
  1 &0 \\ 0 & 0
\end{pmatrix}$, $e_{22}:=
\begin{pmatrix}
  0 & 0\\ 0 & 1
\end{pmatrix}$. Note that the presence of $\alpha^*$ implies that $F_X(t)$ is an operator on $L^2(\slashed{S}^+)$.
\begin{definition}\label{def:fund_class}
  In the above situation, the (locally finite) fundamental class $[D_X] \in
  K_0(C^*_{L,d}(X))=K_0^L(X)$ is given explicitly by $[P_X] - [e_{11}]$, with
$$P_X:=\begin{pmatrix}
  FF^* + (1-FF^*)FF^* & F(1-F^*F)+ (1-FF^*)F(1-F^*F)\\
  (1-F^*F)F^* & (1-F^*F)^2
\end{pmatrix}.
$$
In this formula $F$ denotes $F_X(\cdot)$ and $P_X$ is an idempotent in
$M_2(C^*_{L,d}(X)^+)$. Here,  $A^+$ denotes the unitalisation of $A$. 
\end{definition}
\begin{remark}
 Note that since $\Psi_t$ is assumed to have compactly supported Fourier
 transform, $\Psi_t(D_X)$ has finite propagation which means that
 $P_X$ is a matrix over the unitalisation of $\mathbb{R}^*_L(X) \subset C^*_{L,\maxi}(X)$.
\end{remark}

Now let $M$ be a compact spin manifold with boundary $N$. Denote by $N_\infty$ the cylinder $N\times [0,\infty)$ and by $M_\infty$ the manifold
$M\cup_{N}N_\infty$. Given a Riemannian metric on $M$ which is collared at the
boundary, we will equip $N_{\infty}$ with the product metric. 
Taking the image of $[D_{M_\infty}]$ in $K^L_*(M_\infty,N_\infty)$ and
then under the excision isomorphism defines the relative fundamental class
$[M,N]\in K^L_*(M,N)$. For the index calculations which we have to carry out we need
an explicit representative of this class, and this in the model of relative
K-homology as the K-theory of the mapping cone algebra
$C_{\iota(M,N)}$. Therefore, we recall the construction of \cite{CWY},
referring for further details to \cite{CWY} ---see also
\cite{HR}*{Proposition 4.8.2} and \cite{HR}*{Proposition 4.8.3}.

As the relative K-homology groups are
constructed as mapping cones which come with a built-in shift of degree, we
have to use Bott periodicity to shift the fundamental class to the suspension
algebra (with degree shift). To implement this, denote by $v$ the Bott
generator of $K_1(C_0(\mathbb{R}))$.
Following \cite{CWY} define the invertible element
\begin{equation*}
\tau_D:= v \otimes P_{M_{\infty}} + I \otimes (I-P_{M_{\infty}})
\end{equation*}
in a matrix algebra over $C(S^1) \otimes C^*_{L,d}(M_\infty)^+$ with inverse given by $\tau_D^{-1} =
v^{-1} \otimes P_{M_\infty} + I \otimes (I-P_{M_\infty})$ (see \cite{HR}*{Proposition 4.8.3} for more details). Next, we map to
the relative K-homology of the pair $(M,N)$, which requires applying the
inverse of the excision isomorphism $K_*(M,N)\to
K_*(M_\infty,N_\infty)$. This is implemented for our K-theory cycles by
multiplication with a 
cut-off. For technical reasons, we observe that instead of $N\subset M$ we can
use the homeomorphic $N_R:=N\times \{R\}\subset M_R:= M\cup N\times [0,R]$ for
each $R\ge 0$. We use localisation algebras, and then we can use the K-theory
isomorphism $C^*_{L,d}(M_R)\to C^*_{L,d}(M\subset M_\infty)$ and work with
$C^*_{L,d}(M\subset M_\infty)$ which is independent of $R$. Similarly, we use the
K-theory isomorphism $C^*_{L,d}(N_R)\to C^*_{L,d}(N\subset N_\infty)$ and replace
$C^*_{L,d}(N_R)$ by the $R$-independent $C^*_{L,d}(N\subset N_\infty)$. This causes
slight differences to the construction of \cite{CWY}.

For the cut-off, set $\chi_R:=\chi_{M_R}$, the characteristic function of $M_R$.
Consider
\begin{equation*}
  \begin{split}
    \tau_{D,R}:=& v\otimes (\chi_R P_{M_\infty}\chi_R +
    (1-\chi_{R})\,e_{11}(1-\chi_{R})) \\
     + & I \otimes (I-
    (\chi_RP_{M_\infty}\chi_R + (1-\chi_{R})e_{11}(1-\chi_{R})) )
  \end{split}
\end{equation*}
and define $\tau_{D,R}^{-1}$ in the same way with $v$ replaced by
$v^{-1}$. Note that these two operators are in general \emph{not} inverse to
each other. Define, for $s \in [0,1]$,
{\small
  \begin{equation*}
w_{D,R}(s):=\begin{pmatrix}
 I & (1-s)\tau_{D,R}\\
 0 & I
\end{pmatrix}
\begin{pmatrix}
 I & 0\\
 -(1-s)(\tau_D)^{-1}_M & I
\end{pmatrix}
\begin{pmatrix}
 I & (1-s)\tau_{D,R}\\
 0 & I
\end{pmatrix}
\begin{pmatrix}
 0 & -I\\
 I & 0\\
\end{pmatrix}
.
\end{equation*}
}
Finally set
\begin{equation}\label{eq:def_of_q}
q_{D,R}(s) := w_{D,R}(s)\;e_{11}\, w_{D,R}(s)^{-1} .
\end{equation}

Applying the same procedure not to $\tau_D$ but to 
$v \otimes e_{11} + I \otimes  e_{22}$, we obtain a curve
$q_p(s)$. Note that by construction of $\tau_{D,R}$, all operators, in
particular $q_{D,R}(s)$, are diagonal for the decomposition
$L^2(M_\infty)=L^2(M_R)\oplus L^2(N\times [R,\infty))$ and are of standard
form on 
$L^2(N\times [R,\infty))$. This summand does not appear in \cite{CWY} but has
to be there to construct the appropriate operators in $C^*_{L,d}(M\subset
M_\infty)$.

\begin{lemma}\label{lem:form_of_ops}
  Assume that the operator $F_{M_\infty}(t)$ has propagation $\le L$ for some
  $L\in[0,\infty)$. Then $q_{D,R}(s)(t)$ (recall that we always 
  have an additional $t\in [1,\infty)$-dependency) has propagation $\le
  30L$. It is diagonal with respect to the decomposition
  $L^2(M_\infty)=L^2(M_R)\oplus L^2(N\times [R,\infty))$ and coincides with
  $q_p(s)$ on $L^2(N\times [R,\infty))$. It is obtained via finitely many
  algebraic operations (addition, composition) from $\Psi_t(D_{M_\infty})$, the measurable
  bundle isometry $\alpha$, the
  Bott element $v$ and $\chi_R$.

  If $R>30L$ then $q_{M,R}(0)(t)$ differs from $q_p(0)(t)$ by an operator $Q$
  supported on $N\times [0,R]$. More precisely, for suitable operators $A,B$,
  \begin{equation*}
    Q= \chi_R A \circ I\otimes [\chi_R,P_\infty]\circ B\chi_R
  \end{equation*}
  where the commutator $[\chi_R,P_\infty]$ is supported on $N\times
  [R-5L,R+5L]$ and $Q$ has propagation $\le 30L$.

  Like $q_{D,R}(s)(t)$, the operator $Q(t)$ is obtained via finitely many
  algebraic operations from $\Psi_t(D_{M_\infty})$, $\alpha$, $v$, $v^{-1}$, and $\chi_R$.

  Due to the local nature of all constructions and because of the support
  property of the commutator $[\Psi_R,P_\infty]$ (using Lemma \ref{relind} for
  $\Psi_t(D)$), the operator $Q$ on $L^2(N\times [0,R])$ is \emph{equal} to
  the operator constructed correspondingly,
   where $D_{M_\infty}$ is replaced by
  $D_{N\times\reals}$ and $\chi_R$ by $\chi_{N\times (-\infty,R]}$. 
\end{lemma}
\begin{proof}
  The explicit formulas show that $q_{D,M}(s)(t)$ is an algebraic combination
  of $\Psi_t(D_{M_\infty})$, $\alpha$, etc.~as claimed, where all building
  blocks either have propagation $0$ or are $\Psi_t(D_{M_\infty})$, and we
  compose at most $30$ of the latter. The claim about the propagation
  follows.

  As it can be seen from the formula in the proof of \cite{CWY}*{Claim 2.19},
  $q_{D,R}(0)$ would be equal to $q_p(0)$ if $\tau_{D,R}$ was
  invertible with inverse $\tau_{D,R}^{-1}$, which would happen if
  $\chi_RP_{M_\infty}\chi_R$ was an idempotent. To compare with this situation
  one has to commute $P_{M_\infty}$ and $\chi_R$ which produces the shape of
  $Q$ as claimed. The rest then follows as for
  $q_{D,R}(s)$. 
\end{proof}

Denote by $\iota_R^\prime$ the inclusion of $C^*_{L,d}(N\subset N_\infty)$ in
$C^*_{L,d}(M\subset M_\infty)$, the image consisting of those operators which act
only on $L^2(N_\infty)$.

The relative fundamental class $[M,N] \in K_0(C_{S\iota_R^\prime}) \cong
K_0(SC_{\iota^\prime}) \cong K^L_0(M,N)$ is defined as
\begin{equation}\label{eq:def_of_rel_fundclass}
[M,N]:=\left[(q_{D,R}(0),q_{D,R}(\cdot))\right] -
\left[(q_p(0),q_p(\cdot))\right].
\end{equation}
It is implicit in \cite{CWY} that the K-theory class is independent of $R$ and
the family of normalising functions $\Psi_t$.

\begin{definition}[The Relative Index] \label{defrelind}
 The relative index of the Dirac operator is defined as 
$$\mu([M,N]) \in K_0(C^*_{d}(\pi_1(M),\pi_1(N))).$$ 
The explicit K-theory cycle defining $[M,N]$ and the description of
 the map $\mu$ of Definition \ref{def:relative_index} gives us an explicit
 cycle for the relative index:

 We have to lift the operators
 $q_{D,M}(s)$ 
 involved in the construction of $[M,N]$ to equivariant operators on the
 $\pi_1(M)$-cover $\widetilde M_\infty$ and those involved in $q_{D,M}(0)$ to equivariant
 operators on the
 $\pi_1(N)$-cover $\tilde N_\infty$. This is 
 possible here and the operators are given as the corresponding functions of
 the Dirac operator on the coverings. For this, we use that by Lemma
 \ref{lem:form_of_ops} the operators $q_{D,M}(t)$ is
 obtained as an expression in functions of the Dirac operator which lift to
 the corresponding  functions of the Dirac operator by Lemma \ref{relind}.

 Similarly, by Lemma \ref{lem:form_of_ops} and if $R>30L$, where the
 propagation of $\Psi_t(D)$ is bounded by $L$ for all $t\in [1,\infty)$, the
 operator $q_{D,R}(0)$ is obtained as an algebraic combination of functions of
 $D_{N\times \reals}$ and the cut-off function $\chi_{N\times (-\infty,R]}$
   which lift by Lemma \ref{relind} to $\pi_1(N)$-equivariant operators on
   $\tilde N\times [0,\infty)$ defined by the same expressions. Thus if we denote by $\tilde q_{D,R}$ the element
   constructed as above using the Dirac operator of $\widetilde{M_\infty}$ and $\chi_{\widetilde{M_R}}$ and by
   $\tilde q^N_{D,R}$ the element constructed using the Dirac operator on $\widetilde{N}\times \reals$ and
   $\chi_{N\times (-\infty,R]}$ then we have the following
   \begin{lemma}
    The expression $[(\tilde q^N_{D,R}(0), \tilde q_{D,R}(\cdot))] -
     [(q_p(0),q_p(\cdot))]$ defines an element of $K_0(SC_{C^*_{L,d}(\tilde N\subset \tilde
       N_\infty)^{\pi_1(N)}\to C^*_{L,d}(\widetilde M\subset \widetilde M_\infty)^{\pi_1(M)}})$ which identifies under the canonical
       isomorphism of the latter group
       with $K_0(M,N)$ with $[M,N]$.
   \end{lemma}

   Hence under these conditions on $R$ and the propagation of $\Psi_t(D)$, the
   relative index is the obtained by evaluation at $t=1$, or by homotopy
   invariance at any $t\ge 1$:
   \begin{align}\label{eq:cycle_for_relind}
   \begin{split}
     &\mu([M,N]) = [(\tilde q^N_{D,R}(0)(t), \tilde q_{D,R}(\cdot)(t))] -
     [(q_p(0),q_p(\cdot))] \in \\
     & K_0(SC_{C^*(\tilde N\subset \tilde
       N_\infty)^{\pi_1(N)}\to C^*_{d}(\widetilde M\subset \widetilde M_\infty)^{\pi_1(M)}})
     \iso K_0(C^*_{d}(\pi_1(M),\pi_1(N))).
     \end{split}
   \end{align}
   As $q_p(\cdot)$ is independent of $t$, we omit specifying the evaluation at
   $t$ here. 
\end{definition}

\subsection{The Localised Fundamental Class and Coarse Index} \label{locind}
Suppose $X$ is a smooth even dimensional spin
manifold with free and proper action by $\Gamma$. Let $Z$ be a closed
$\Gamma$-invariant subset 
of $X$. Suppose that there exists a complete $\Gamma$-invariant Riemannian
metric on $X$ which has
uniformly positive scalar curvature outside $Z$. In \cite{RI} and in more
detail in \cite{RP}, Roe defines a localised coarse index of the Dirac
operator in $K_*(C^*_{\red}(Z \subset X)^\Gamma)$. In the course of the proof
of \cite{HankePapeSchick}*{Theorem 3.11}, the construction of the latter
localised index is generalised to the case of a Dirac operator twisted with a
Hilbert $C^*$-module bundle. In \cite{RZ}*{Chapter 2}, Zeidler defines this
index 
using localisation algebras. There, he also shows that under certain
assumptions on a manifold $X$ with boundary $Y$, the localised coarse index
can be used to define an obstruction to the extension of a uniformly positive
scalar curvature metric on the boundary to a uniformly positive scalar
curvature metric on the whole manifold. In this section we follow the
approach in \cite{RZ} to define the localised fundamental class and coarse
index.

\begin{definition}
 Denote by $C^*_{L,0,d}(X)^\Gamma$ the kernel of the evaluation
 homomorphism $\ev_1\colon C^*_{L,d}(X)^\Gamma \rightarrow C^*_{d}(X)^\Gamma$.
 Denote by $C^*_{L,Z,d}(X)^\Gamma$ the preimage of $C^*_{d}(Z \subset
 X)^\Gamma$ under $\ev_1$. The symbol $d$ here stands for the chosen
 completion ($\red$, $\max$, or $q$). 
\end{definition}

Suppose that $g$ is a $\Gamma$-invariant metric on $X$ with uniformly positive scalar curvature outside of a $\Gamma$-invariant set $Z$. In \cite{RZ}*{Definition 2.2.6}, in
this situation the so-called partial $\rho$-invariant
$\rho^\Gamma_{Z,\red}(g) \in C^*_{L,Z,\red}(X)^\Gamma$ is constructed, which
is a lift of $[D_X]$ 
under the morphism $K_*(C^*_{L,Z,\red}(X)^\Gamma) \rightarrow
K_*(C^*_{L,\red}(X)^\Gamma)$ induced by the inclusion.

Recall the explicit representative for $[D_X] \in K_0(C^*_{L,d}(X)^\Gamma)$
of Section \ref{relindd}. We next recall the
construction of \cite{RZ}*{Definition 2.2.6} and show that it also works for
$C^*_q$.

\begin{lemma}\label{relind}
  If $f_2\in C_b(\reals)$ has Fourier transform with
support in $[-r,r]$ then $f_2(D)$ is $r$-local and  depends only on the
\emph{$r$-local 
  geometry} in the following sense: if $A\subset X$ is a $\Gamma$-invariant
measurable subset then $\chi_Af_2(D)(1-\chi_{B_r(A)})=0$ and $\chi_Af_2(D)$
depends only on the Riemannian metric on $\overline{B_r(A)}$.
\end{lemma}
\begin{proof}
    This is the usual unit propagation statement in the form
  that $f_2(D)$ is the integral of $\hat f_2(t) e^{itD}$ where $e^{itD}$ not
  only has propagation $|t|$ but also is well known to depend only on the
  $r$-local geometry. The latter fact is a consequence of \cite{HR}*{Corollary 10.3.4}. 
\end{proof}

\begin{lemma}[\cite{RP}*{Lemma 2.3}, \cite{HankePapeSchick}*{Proposition 3.15}] \label{relind2}
 Suppose as above that the scalar curvature of $g$ outside $Z$ is
 bounded from below by
 $4\epsilon^2$. If $f \in C_0(\mathbb{R})$ has support in $(-\epsilon,\epsilon)$, then $f(D)$ lies in
 $C^*_{d}(Z \subset X)^\Gamma$.
\end{lemma}
\begin{proof}
  By \cite{HankePapeSchick}*{Proposition 3.15} the statement holds for all
  quotients $X/N$ and their reduced Roe algebra, which implies by definition
  of the quotient completion that it holds for $C^*_q(X)^\Gamma$.
\end{proof}

Because of the geometric nature of the completion of the Roe algebra we use,
Lemmas \ref{relind} and \ref{relind2} allow to define the localised coarse
index using the
completion $C^*_q$ as follows. 
\begin{definition}\label{def:partialrho}
  Choose a sup-norm continuous family of normalising functions $\Psi_t$ for
  $t\geq 1$ 
  such that $\Psi_1^2-1$ has support in $(-\epsilon,\epsilon)$, the
  Fourier transform of $\Psi_t$ has compact support for each $t>1$ and
  the Fourier transform of $\Psi_t$ has support in
  $[-\frac{1}{t},\frac{1}{t}]$ for $t\geq 2$. Note that the support condition
  on $\Psi_1$ implies that its Fourier transform is \emph{not} compactly
  supported. For the existence note that we have to approximate the Fourier
  transform of $\Psi_1$ by compactly supported functions (with a
  singularity at $0$) such that the error is small in $L^1$-norm. This is
  possible, as can be seen from the discussion in the proof of \cite{HankePapeSchick}*{Lemma 3.6}. 

  Define $F_X(t)$ and $P_X$ as in Section \ref{relindd}.
  Observe, however, that by Lemma \ref{relind} $F_X(1)F_X(1)^*-1 \in
  C^*_{q}(Z\subset X)^\Gamma$. It follows that now the cycle
  $[P_X]-[e_{11}]$ defines a class
  $$\rho^\Gamma_Z(g) \in K_0(C^*_{L,Z,d}(X)^\Gamma)$$
  which is of course a lift of $[D_X]$.
\end{definition}

\begin{corollary}\label{corol:lift}
  The construction shows that if we have uniform positive scalar
  curvature not only on $X\setminus Z$ but on all of $X$ there is a further lift of
  $\rho^\Gamma_Z(g)$ to
  $\rho^\Gamma(g)\in K_0(C^*_{L,0,d}(X)^\Gamma)$, the usual rho-invariant.
\end{corollary}

\begin{definition}\label{def:coarse_ind}
 Let $Z \subset X$ and $g$ be as above. Suppose furthermore that the action of $\Gamma$ on $Z$ is cocompact so that Lemma~\ref{relroei} holds for $Z$. The
 equivariant localised coarse index $\Ind_Z^\Gamma(g)$ of $g$ with respect to $Z$ is defined as the image of $\rho^\Gamma_Z(g)$ under the composition
 $$K_0(C^*_{L,Z,d}(X)^\Gamma) \rightarrow K_0(C_{d}^*(Z \subset X)^\Gamma) \cong K_0(C_{d}^*(Z)^\Gamma),$$
 where the first map is induced by evaluation at $1$.
\end{definition}
The long exact sequence in $K$-theory associated to the short exact sequence
$$ 0 \rightarrow C^*_{L,0}(X)^{\Gamma} \rightarrow C^*_{L,Z}(X)^{\Gamma} \rightarrow C^*(Z \subset X)^{\Gamma} \rightarrow 0,$$
along with Corollary \ref{corol:lift} imply that if
$g$ has uniformly positive scalar curvature on all of $X$, then
$\Ind_Z^\Gamma(g)$ vanishes.

\subsection{Application to the Case of a Compact Manifold with Boundary}

Suppose $M$ is compact even-dimensional spin manifold with boundary $N$. In
this case we cannot directly define an index for the Dirac operator on $M$
with value in $K_*(C_q^*(\pi_1(M)))$. However given a
metric $g$ with positive scalar curvature and product structure near the
boundary, we can use the above localised coarse index to define an index in
$K_0(C_q^*(\widetilde{M})^{\pi_1(M)}) \cong K_0(C_q^*(\pi_1(M)))$. Note
that this index \emph{does} in general depend on the chosen metric of positive
scalar 
curvature near the boundary. Let us review the construction of
the latter index. 

As in Section~\ref{relindd}, denote by $N_\infty$ the cylinder $N\times [0,\infty)$ and by $M_\infty$ the manifold $M\cup_{N}N_\infty$.
Denote by $[D_{M_{\infty}}]$ the fundamental class of the Dirac operator in $K_*(C^*_{L,q}(M))$ associated to some metric $g$ on $M_\infty$ (not necessarily collared on the cylindrical end) and by
$[\widetilde{D_{M_{\infty}}}]$ the fundamental class of the Dirac operator in $K_*(C^*_{L,q}(\widetilde{M})^{\pi_1(M)})$ on $\widetilde{M_\infty}$ associated
to the pullback of $g$, which we denote by $\tilde{g}$. As observed in Remark
\ref{rem:ext_of_prop_lift}, Proposition~\ref{eps} extends to $M_\infty$ and the
pointwise lifting procedure of operators with small propagation gives rise to
an isomorphism
$K^L_*(M_\infty) \cong K_*(C^*_L(\widetilde{M_\infty})^{\pi_1(M)})$ under which $[D_{M_{\infty}}]$ is mapped to $[\widetilde{D_{M_{\infty}}}]$.
If $g$ has positive scalar curvature on $N$, then its pullback has uniformly positive scalar
curvature on $N_\infty^{\prime} \subset \widetilde{M_\infty}$, i.e.~outside
the cocompact subset $\widetilde{M}$ of $\widetilde{M_\infty}$. This allows us to
the define the localised coarse index
$\Ind^{\pi_1(M)}(g):=\Ind_{\widetilde{M}}^{\pi_1(M)}(\tilde{g}) \in
K_0(C^*(\widetilde{M})^{\pi_1(M)}) \cong K_0(C^*(\pi_1(M)))$. The
latter index is an obstruction to $\tilde{g}$, and thus $g$, having positive
scalar curvature.

\section{Statement and Proof of the Main Theorem}
\label{sec:proof}

Finally we are in the position to state the main theorem of this
paper. Throughout this section we will assume all the manifolds and their
boundary to be path-connected.
\begin{theorem} \label{mainthm}
 Let $M$ be a compact spin manifold with boundary $N$. We have the commutative diagram
 $$\begin{tikzcd}
\to   K^L_*(N) \arrow{r} \arrow{d}{\mu_N} & K^L_*(M) \arrow{r} \arrow{d}{\mu_M} & K^L_*(M,N)
   \arrow{d}{\mu} \to  \\
 \to  K_*(C_q^*(\pi_1(N))) \arrow{r} & K_*(C_q^*(\pi_1(M))) \arrow{r}{j}
   & K_*(C^*_q(\pi_1(M),\pi_1(N))) \to
 \end{tikzcd}$$
 where the vertical maps are the index maps and relative index maps.
 
Assume that $M$ has a metric $g$ which is collared at the boundary and
 has positive scalar curvature there. Then
\begin{equation*} j( \Ind^{\pi_1(M)}(g)) = \mu([M,N]) 
\end{equation*}
under the canonical map $ j\colon K_*(C_q^*(\pi_1(M))) \rightarrow K_*(C_q^*(\pi_1(M),\pi_1(N)))$.
\end{theorem}

The above theorem has as a corollary the following vanishing theorem of Chang, Weinberger and Yu for the relative index constructed in the mapping cone
of the quotient completion of the group ring:
\begin{theorem}
 Let $M$ be a compact spin manifold with boundary $N$. Suppose that $M$ admits a metric of uniformly positive scalar curvature which is collared at the
 boundary. Then $\mu([M,N]) = 0$.
 \label{TCWY}
\end{theorem}

\begin{proof}[Proof of the Theorem~\ref{mainthm}]
Proposition \ref{prop:Lift_compatible_with_Y_to_X} implies the commutativity of the diagram. To see this, note that the discussion there relies only on the functoriality properties of the maximal completions which are also satisfied by the quotient completions. It remains to show that given a metric with positive scalar curvature at the boundary,
 $\Ind^{\pi_1(M)}( g)$ is mapped to $\mu([M,N])$ under the canonical map. Let us
 analyse the situation with the strategy of proof and the difficulties involved. 
 For the notation used we refer to Sections~\ref{relindd} and \ref{locind} on
 the relative index and the localised coarse index.

 Both index classes are defined using explicit expressions involving functions
 of the Dirac operator. For $\Ind^{\pi_1(M)}( g)$, we only use the manifold $\widetilde M$
 and $\pi_1(M)$-equivariant constructions, which, however, are necessarily
 non-local to make use of the invertibility of the Dirac operator on the
 boundary. For $\mu([M,N])$, on the other hand, one has to use a
 $\pi_1(M)$-equivariant operator on $\widetilde M$ and a further lift to a
 $\pi_1(N)$-equivariant operator on $\tilde N$, which is only possible if all
 the functions of the Dirac operator involved are sufficiently local. To show
 that the two classes are mapped to each other, we need to reconcile these two
 points.

 First, observe that in the construction of the relative fundamental class and
 relative index we use the explicit implementation of the Bott periodicity
 map. We apply this now to our representative of the local index: with our
 choice of $\Psi_1$, $P_{\widetilde M_\infty}(1)$ is an idempotent in $C_q^*(\tilde 
 M\subset \widetilde M_\infty)^{\pi_1(M)}$ representing
 $\Ind^{\pi_1(M)}(g) \in K_0(C_q^*(\widetilde
 M\subset \widetilde M_\infty)^{\pi_1(M)})\iso
 K_0(C^*(\pi_1(M)))$. Next,
 \begin{equation*}
   \tau := v\otimes P_{\widetilde M_\infty}(1) + I\otimes (I-P_{\widetilde M_\infty}(1))
 \end{equation*}
  is the invertible element in $C_0(\reals)\otimes C^*(\widetilde M\subset \tilde
  M_\infty)^{\pi_1(M)}$ representing the $K_1$-class corresponding to the
  localised index under the suspension isomorphism. Finally, if we define
  $q(s)$ as in Equation \eqref{eq:def_of_q} with $\tau_{D,R}$ replaced by
  $\tau$ then 
  \begin{equation*}
   a:= [q(0)(1),q(\cdot)(1)]-[q_p(0),q_p(\cdot)] \in K_0(SC_{\{0\}\to C^*(\tilde
      M\subset \widetilde M_\infty)^{\pi_1(M)}})
  \end{equation*}
  defines the class corresponding to $\Ind^{\pi_1(M)}( g)$
  under the Bott periodicity isomorphism, where we use that the cone of the
  inclusion of $\{0\}$ into $A$ is the suspension of $A$. Of course, here
  $q(0)(1)=q_p(0)$. 

  We now have to show that, under the canonical map to the suspension of the
  cone of $C^*(\tilde
    N\subset \tilde N_\infty)^{\pi_1(N)}\to C^*(\widetilde M\subset \tilde
    M_\infty)^{\pi_1(M)}$ induced by the inclusion $\{0\}\to C^*(\tilde
  N,\tilde N_\infty)^{\pi_1(N)}$, the class $a$ is mapped to the relative index
  $\mu[M,N]$. Recall from \eqref{eq:cycle_for_relind} that the latter is
  represented by any cycle
  of the form
  \begin{equation*}
    [\tilde q^N_{D,R_t}(0)(t),\tilde q_{D,R_t}(\cdot)(t)] - [(q_p(0),q_p(\cdot))]
  \end{equation*}
  for $t>1$, such that the support of $\widehat\Psi_t$ is contained in
  $[-L_t,L_t]$ for $L_t\in\reals$ and therefore $\Psi_t(D)$ has propagation
  $\le L_t$, where we must choose $R_t> 30 L_t$. The construction of $\tilde
  q_{D,R_t}(\cdot)(t)$ involves the same steps as the one of $q(\cdot)$, but
  we use $\Psi_t(D)$ instead of $\Psi_1(D)$ and moreover apply cut-off with
  $\chi_{R_t}$. Note that now $\tilde q^N_{D,R_t}(0)(t)-q_p(0)\ne 0$, but rather
  $\tilde q^N_{D,R_t}(0)(t)-q_p(0)\in C^*(\tilde N\subset \tilde
  N_\infty)^{\pi_1(N)}$, so that this is not a class in the suspension of
  $SC^*(\widetilde M\subset \widetilde M_\infty)^{\pi_1(M)}$ but in the mapping cone.

  We claim now that for each $\epsilon>0$ there is $(t_\epsilon,R_\epsilon)$
  such that
  \begin{equation}\label{eq:norm_conv}
    \norm{\tilde q^N_{D,R_\epsilon}(0)(t_\epsilon) - q_p(0)} + \norm{\tilde
      q_{D,R_\epsilon}(\cdot)(t_\epsilon) - q(\cdot)(1)} \le \epsilon.
  \end{equation}
  This implies by standard properties of the K-theory of Banach algebras the
  desired result (as $q(0)(1)=q_p(0)$),
  \begin{equation*}
    \mu([M,N]) = c(\Ind^{\pi_1(M)}( g)).
  \end{equation*}

  To prove \eqref{eq:norm_conv} we make use of Lemma \ref{lem:form_of_ops}
  which explicitly describes the operators involved. This implies
  \begin{equation}\label{eq:t_to_1_conv}
    \norm{\tilde q_{D,R}(\cdot)(t) - \tilde q_{D,R}(\cdot)(1)}
    \xrightarrow{t\to 1} 0
  \end{equation}
  uniformly in $R$, as the two expressions are obtained via algebraic
  operations involving $\Psi_t(D)$, and by the sup-norm continuity of
  $\Psi_t$, $\Psi_t(D)$ converges to $\Psi_1(D)$ in norm (and this again
  uniformly, independent of the complete Riemannian manifold for which $D$ is
  considered).

Next by the uniformly positive scalar curvature on $N_\infty$ we have
$P_{\widetilde M_\infty}(1)-e_{11} \in C^*(\widetilde M\subset \widetilde
M_\infty)^{\pi_1(M)}$. This implies (convergence in norm)
\begin{equation*}
  \chi_R (P_{\widetilde M_\infty}(1)-e_{11})\chi_R \xrightarrow{R\to\infty}
  P_{\widetilde M_\infty}(1)-e_{11}
\end{equation*}
or equivalently
\begin{equation}
  \label{eq:projector_conv}
  \chi_R P_{\widetilde M_\infty}(1)\chi_R + (1-\chi_R)\,e_{11}\,(1-\chi_R)
  \xrightarrow{R\to\infty} P_{\widetilde M_\infty}(1).
\end{equation}

Because of Lemma \ref{lem:form_of_ops}, \eqref{eq:projector_conv} implies that
\begin{equation}
\norm{\tilde q_{D,R}(\cdot)(1)-q(\cdot)(1)}\xrightarrow{R\to\infty}
0\label{eq:conv_no_cut}
\end{equation}

as these operators are obtained as a fixed algebraic expression of either
\begin{equation*}
\chi_R
P_{\widetilde M_\infty}(1)\chi_R + (1-\chi_R)\,e_{11}\,(1-\chi_R)\quad\text{
  or }
P_{\widetilde M_\infty}(1).
\end{equation*}

Next, \eqref{eq:t_to_1_conv} together with \eqref{eq:conv_no_cut} imply the
assertion of \eqref{eq:norm_conv} for the second summand. Here, we can and
have to choose $R_\epsilon$ depending on $t_\epsilon$ such that
$R_\epsilon>R_{t_\epsilon}$ (depending on the propagation of
$\Psi_{t_\epsilon}(D)$).

Then, the lift $\tilde q^N_{D,R_\epsilon}(0)(t_\epsilon)$ to $C^*(\tilde
  N\subset \tilde N_\infty)^{\pi_1(N)}$ actually exists, is defined in terms
  of the Dirac operator on $\tilde N\times \reals$, and we have to show that
  by choosing $t_\epsilon$ sufficiently close to $1$ it is close to $q_p(0)$.

  This, as we already showed, it is a special case of \eqref{eq:t_to_1_conv}
  and \eqref{eq:conv_no_cut}, now
  applied to the Dirac operator on $\tilde N\times \reals$. Note that because
  of the
  invertibility of the Dirac operator on $N\times \reals$ and our appropriate
  choice of the normalising function $\Psi_1$, we have on the nose
  \begin{equation*}
    \tilde q^N(0)(1) = q_p(0),
  \end{equation*}
  where $q^N$ is defined like $q$ but using the Dirac operator on $\tilde
  N\times \reals$.
This finishes the proof of \eqref{eq:norm_conv} and therefore of our main
Theorem \ref{mainthm}.
\end{proof}

\begin{remark}\label{rem:odd_dimensions}
  We decided to present the details of the index constructions and proofs only
  for even dimensional manifolds.

  The case of odd dimensional manifolds can easily be reduced to this case via
  a ``suspension construction'', as also done in \cite{CWY}. More precisely,
  if we have an odd dimensional compact manifold $M$, we pass to the even
  dimensional manifold $M\times S^1$. Correspondingly, the covering space
  $\tilde M$ with action by $\pi_1(M)$ is replaced by $\tilde M\times \reals$
  with action of $\pi_1(M)\times\integers$.

  It is now a standard result that we have K\"unneth isomorphisms for the
 K-theory groups relevant to us, in particular for a group homomorphism
 $\Lambda\to \Gamma$
  \begin{equation}\label{eq:Kuen1}
    K_0(C^*_{d}(\Gamma\times\integers,\Lambda\times\integers))\xrightarrow{\iso}
    K_0(C^*_{d}(\Gamma,\Lambda))\oplus K_1(C^*_{d}(\Gamma,\Lambda)).
  \end{equation}
The ad hoc definition of the relative index $\mu(M,N)\in
K_1(C^*_{d}(\pi_1(M),\pi_1(N))$, generalizing Definition \ref{defrelind} to odd
dimensional $M$, is now just the image of $\mu([M\times S^1,N\times S^1])$
under the K\"unneth map \eqref{eq:Kuen1} (and indeed, the $K_0$-component is
zero).

Because positive scalar curvature of $M$ implies positive scalar curvature of
$M\times S^1$, Theorem \ref{TCWY} for odd dimensiona $M$ follows from its
version for the even dimensional $M\times S^1$.

In the same way, using K\"unneth and suspension isomorphisms for the whole
diagram of Theorem \ref{mainthm} (using along the way e.g.~\cite[Section
5]{RZA}), the statement and proof of Theorem \ref{mainthm} for odd dimensional
$M$ follows from the corresponding one for the even dimensional $M\times S^1$.

More systematically, Zeidler \cite{RZA} develops a setup of 
$Cl_n$-linear Roe algebras and localisation algebras and $Cl_n$-equivariant
Dirac operators on $n$-dimensional spin manifolds. Our constructions and
arguments should carry through in this setup, given a uniform treatment for
all dimensions, and working with real group $C^*$-algebras. As this requires a
bit more notation and additional concepts, and as we were striving for a down
to earth exposition, we decided to stick to the
classical setup and leave it to 
the interested reader to work out the details of such an approach.
\end{remark}


\if{0=1}
\bibliography{references}
\fi

\end{document}